\documentclass{amsart}
\usepackage{amssymb,latexsym}
\usepackage[usenames,dvipsnames,svgnames,table]{xcolor}
\usepackage{graphicx,enumerate}
\usepackage{hyperref}
\usepackage{algorithm}
\usepackage{algorithmic}
\usepackage{boxedminipage}
\usepackage{float}
\usepackage{longtable}
\usepackage[utf8x]{inputenc}
\usepackage{xy}
\usepackage[T1]{fontenc}
\DeclareGraphicsExtensions{.png, .pdf}

\makeatletter
\newenvironment{breakablealgorithm}
  {
   \begin{center}
     \refstepcounter{algorithm}
     \hrule height.8pt depth0pt \kern2pt
     \renewcommand{\caption}[2][\relax]{
       {\raggedright\textbf{\ALG@name~\thealgorithm} ##2\par}%
       \ifx\relax##1\relax 
         \addcontentsline{loa}{algorithm}{\protect\numberline{\thealgorithm}##2}%
       \else 
         \addcontentsline{loa}{algorithm}{\protect\numberline{\thealgorithm}##1}%
       \fi
       \kern2pt\hrule\kern2pt
     }
  }{
     \kern2pt\hrule\relax
   \end{center}
  }
\makeatother

\theoremstyle{plain}
\newtheorem{theorem}{Theorem}

\newtheorem{proposition}[theorem]{Proposition}
\newtheorem{lemma}[theorem]{Lemma}
\theoremstyle{definition}
\newtheorem{definition}[theorem]{Definition}
\newtheorem{example}[theorem]{Example}
\newtheorem{notation}[theorem]{Notation}
\newtheorem{remark}[theorem]{Remark}

\definecolor{cornellred}{rgb}{0.7, 0.11, 0.11}
\definecolor{payne\'sgrey}{rgb}{0.25, 0.25, 0.28}
\definecolor{trolleygrey}{rgb}{0.5, 0.5, 0.5}
\definecolor{sapgreen}{rgb}{0.31, 0.49, 0.16}
\newcommand{\red}[1]{{\color{red}#1}}

\newcommand{\grey}[1]{{\color{payne\'sgrey}#1}}
\newcommand{\green}[1]{{\color{sapgreen}#1}}
\newcommand{\cornellred}[1]{{\color{cornellred}#1}}

\newcommand{\Gr}[2]{\mathbb{G}({#1},{#2})}

\DeclareMathOperator{\img}{img}
\DeclareMathOperator{\Sym}{Sym}

\newcommand{\cC}{\mathcal{C}}
\newcommand{\wh}{\mathbf{h}}
\newcommand{\wv}{\mathbf{v}}
\newcommand{\ww}{\mathbf{w}}
\newcommand{\vh}{\vec{\mathbf{h}}}
\newcommand{\vv}{\vec{\mathbf{v}}}
\newcommand{\PP}{\mathbb{P}}

\begin{document}

\title{Skew-Symmetric Tensor Decomposition}
\author[E. Arrondo]{Enrique Arrondo}
\address{Departamento de \'{A}lgebra, Geometr\'{\i}a y Topolog\'{\i}a, Facultad de Ciencias Matem\'{a}ticas, Universidad Complutense de Madrid, 28040 Madrid, Espa\~{n}a}
\email{arrondo@mat.ucm.es}
\author[A. Bernardi]{Alessandra Bernardi}
\address{Dipartimento di Matematica, Universit\`{a} di Trento, Via Sommarive, 14, 38123 Povo, Italia}
\email{alessandra.bernardi@unitn.it}
\author[P. Macias Marques]{Pedro Macias Marques}
\address{Departamento de Matem\'{a}tica, Escola de Ci\^{e}ncias e Tecnologia, Centro de Investiga\c{c}\~{a}o em Matem\'{a}tica e Aplica\c{c}\~{o}es, Instituto de Investiga\c{c}\~{a}o e Forma\c{c}\~{a}o Avan\c{c}ada, Universidade de \'{E}vora, Rua Rom\~{a}o Ramalho, 59, P--7000--671 \'{E}vora, Portugal.}
\email{pmm@uevora.pt}
\author[B. Mourrain]{Bernard Mourrain}
\address{Inria Sophia Antipolis M\'{e}diterran\'{e}e, AROMATH, 2004 route des Lucioles, 06902 Sophia Antipolis, France}
\email{Bernard.Mourrain@inria.fr}

\begin{abstract} We introduce the ``skew apolarity lemma'' and we use it to give algorithms for the skew-symmetric rank and the decompositions of tensors in {$\bigwedge^dV_{\mathbb{C}}$ with $d\leq 3$ and $\dim V_{\mathbb{C}} \leq 8$}. New algorithms to compute the rank and a minimal decomposition of a tritensor are also presented.

\end{abstract}

\maketitle
\section*{Introduction}
The problem of decomposing a structured tensor in terms of its structured rank has been extensively studied in the last decades (\cite{bb13, bb, bbcm, bbo, bdhm, bgi, bknt, ccgo, dot, land, mmsv, oo}). Most of the well-known results are for symmetric tensors, for tensors without any symmetry and for tensors with partial symmetries. In this paper we want to focus on the decomposition of skew-symmetric tensors.

Let $V$ be a vector space of dimension $n+1$ defined over an algebraically closed field $K$ of characteristic zero. Given an element $t\in \bigwedge^d V$, how can we find vectors $v_i^{(j)}\in V$  and $\lambda_i\in K$ in such a way that the following decomposition 
\begin{equation}
t= \sum_{i=1}^r \lambda_i v_i^{(1)}\wedge \cdots \wedge v_i^{(d)}
\end{equation}
involves the minimum possible number of summands up to scalar multiplication?

One can look at this problem from many perspectives. Form the algebraic geometry point of view it corresponds to finding the minimum number $r$ of distinct points on a Grassmannian $\mathbb{G}(d,V)$ whose span contains the given tensor $t$. We call $r$ the \emph{skew-symmetric rank} of the tensor $t$.
From a physical point of view this problem can be rephrased in terms of measurement of the entanglement of fermionic states (see eg. \cite{lh, gr,z, bv}).

The strategy that we will pursue in our manuscript wants to follow the classical algebraic technique that is used for the decomposition of symmetric tensors (which has also a physical interpretation in terms of entanglement of bosonic states, \cite{eb, bc}). Namely, we will define the skew-apolarity action (Section \ref{Preliminaries}) which will allow to build up the most closest concept of an ideal of points that one can have in the skew-symmetric algebra. We will present few examples of skew ``ideals'' of points in Section \ref{a:few:ex}. In Section \ref{3vector} we give our complete analysis for tensors ${t\in \bigwedge^3 V}$ with ${\dim V \leq 8}$.

The idea of finding a skew-symmetric version of apolarity in order to extend some of the results which are known for the symmetric setting to the skew-symmetric one is not new. The novelty of our aproach is the context we use and the way we extend this notion. In \cite{cgg} an analogous of apolarity action is defined: their apolar of a subpace ${Y\subseteq \bigwedge^k V}$ is the subset ${Y^{\perp}\subseteq \bigwedge^{n+1-k} V}$ of elements $w$ satisfiying ${w\wedge t=0}$, for all ${t\in Y}$. That was enough for their purpose since they were interested in the dimension of secant varieties of Grassmannians, while for our purpose we need the whole description of all the elements in $\bigwedge^{\bullet}V$ annihilating the given tensor.  The idea of extending apolarity to contexts different from the classical symmetric one has been pursued by various authors especially in the multi-homogeneous context (e.g. \cite{gal,grv,bbg,be}) but we are not aware of any other apart from \cite{cgg} where it has been investigated in the skew-symmetric context.

We use the skew-apolarity to determine the rank of skew-tensors in $\bigwedge^dV_{\mathbb{C}}$ for ${d\leq 3}$ and ${\dim V \leq 8}$. This is based on the normal form classifications in \cite{Sch,gur1}. The different cases are characterized in terms of the kernels of skew-catalecticant maps, which leads to new algorithms for the decomposition of trivectors in dimension up to $7$.

In the next section, we introduce the notation and prove the skew-apolarity Lemma. In Section ~\ref{a:few:ex}, some examples of ideals of points in the
skew-symmetric algebra are presented and analysed.
In Section \ref{3vector}, some properties of rank-one skew
symmetric tensors are described, that are used in the analysis hereafter.
In Section \ref{sec:trivector}, we give our complete analysis for tensors
${t\in \bigwedge^3 V}$ with ${\dim V \leq 8}$, including algorithms
for their decomposition.

\section{Preliminaries}\label{Preliminaries}

Let us briefly recall what is known in the symmetric case for a minimal decomposition of an element ${f\in S^dV}$ as ${f=\sum_{i_1}^r v_i^{\otimes d}}$.

First of all, observe that we can consider ${S:=\Sym V}$ as a ring of homogeneous polynomials ${S:= K[x_0,\ldots ,x_n]}$ and let ${R:=\Sym V^{*}= K[y_0,\ldots ,y_n]}$ be its dual ring acting on $S$ by differentiation:  
\begin{equation} 
y_j(x_i)=\frac{d}{d x_j}(x_i)=\delta_{ij}. 
\end{equation}
The action above is classically know as \emph{apolarity}.

Denote ${f^{\perp}=\{ g\in R \, | \, g(f)=0 \} \subset R}$ the annihilator of a homogeneous polynomial ${f\in S}$ and remark that it is an ideal of $R$.

\begin{definition}
A subscheme ${X\subset \mathbb{P}(S^1V)}$ is apolar to ${f\in S}$ if its homogeneous ideal ${I_X\subset R}$ is contained in the annihilator of $f$.
\end{definition} 

Useful tools to get the apolar ideal of a polynomial ${f\in S^dV}$ are the well\-\mbox{-known} \emph{catalecticant matrices} which are defined to be the matrices associated to the maps ${C_f^{i,d-i} \in \mathrm{Hom}(S^iV^*, S^{d-i}V)}$, such that ${C_f^{i,d-i}(y_0^{\,i_0}\cdots y_n^{\,i_n})= \frac{\partial^i}{\partial {x_0^{\,i_0} \cdots x_n^{\,i_n} }}(f)}$ with ${\sum_{j=0}^n i_j=i}$, for ${i=0, \ldots , d}$.

The annihilator of a power ${l^d\in S^dV}$ of a linear form ${l\in S^1V}$ is the ideal of the corresponding point ${[l]^{*}\in \mathbb{P}(V^{*})}$ in degree at most $d$.

The following Lemma is classically known as Apolarity Lemma (\cite{IK,dol}).

\begin{lemma}[Apolarity Lemma, \cite{IK,dol}]\label{ApolarityLemma} 
A homogeneous polynomial ${f\in S^dV}$ can be written as ${f=\sum_{i=1}^r a_il_i^{\,d}}$, with ${l_1, \ldots , l_r}$ linear forms, ${a_1,\ldots,a_r\in K}$, if and only if the ideal of the scheme  ${X=\{ [l_1^*], \ldots , [l_r^*]\}\subset \mathbb{P}(R)}$ is contained in $f^{\perp}$.
\end{lemma}

We want to make an analogous construction for the skew-symmetric case.
\begin{definition}\label{skew-apolarity}
Let  ${\wh_{{\{1, \ldots , i\}}}=h_1 \wedge \cdots \wedge h_i\in \bigwedge^iV^*}$ and ${\wv_{\{1, \ldots , i\}}=v_1 \wedge \cdots \wedge v_i \in \bigwedge^iV}$  be two  elements of skew-symmetric rank $1$. For these elements the \emph{skew-apolarity} action is defined as the  determinant among $\bigwedge^i V^*$ and $\bigwedge^iV$:
\begin{equation}\label{dec}
\wh_{{\{1, \ldots , i\}}}(\wv_{{\{1, \ldots , i\}}})= \begin{vmatrix}
h_1(v_1) & \cdots & h_1(v_i)\\
\vdots && \vdots \\
h_i(v_1) & \cdots & h_i(v_i)
\end{vmatrix}
\end{equation}
\end{definition}
This can be compared with the notion of ``flattenings'' introduced by J.M. Landsberg in \cite[Section 3.4]{land}: in the language of that book, the determinant of our Definition \ref{skew-apolarity} as a space of equations is  $\bigwedge^iV^* \otimes \bigwedge^iV^*$.

\begin{definition}
For any ${\wv_{\{1, \ldots , d\}}=v_1\wedge \cdots \wedge v_d}$ in $\bigwedge^d V$ and ${\wh_{\{1, \ldots , s\}}=h_1 \wedge \cdots \wedge h_s}$ in $\bigwedge^sV^*$, with ${0\le s\le d}$, we set
\begin{equation}\label{dot} 
\wh_{\{1, \ldots , s\}} \cdot \wv_{\{1, \ldots , d\}} : = 
\sum_{\substack{R \subset \{1, \ldots , d \} \\ \lvert R\rvert=s}} 
\mathrm{sign}(R) \cdot \wh_{\{1, \ldots , s\}}(\wv_{R})\wv_{\bar{R}},
\end{equation}
where $\wh_{\{1, \ldots , s\}}(\wv_{R})$ is defined as in \eqref{dec} and ${\bar{R}=\{1, \ldots , d\}\setminus R}$. We define the skew-apolarity action extending this by linearity. Now we can define the \emph{skew-catalecticant matrices} ${\mathcal{C}_t^{s,d-s}\in \mathrm{Hom}\big(\bigwedge^sV^*, \bigwedge^{d-s}V\big)}$  associated to any element ${t\in \bigwedge^dV}$ as ${\mathcal{C}_t^{s,d-s}(\wh)= \wh \cdot t}$.
\end{definition}

The skew-symmetric action can be defined intrinsically in terms of co-products and inner product (see \cite{BourbakiAlgebraChapters132008}[A.III, p. 600,603]). See also \cite{DoubiletFoundationsCombinatorialTheory1974} where a geometric calculus is developed using this ``meet'' operator and its dual ``join'' operator.

We would like to thank M. Brion for the following remark.

\begin{remark}\label{Brion} It's worth noting that the above definition of skew-apolarity action is coordinate free. In fact  that action corresponds to the projection of $\bigwedge^sV \otimes \bigwedge^dV$ to $\bigwedge^{d-s}V$ which sends the unique copy of $\bigwedge^dV$ to $\bigwedge^{d-s}V$. Moreover, the fact that there is a unique copy of that irreducile Schur representation also shows that this is the unique way to define such a skew-apolarity action.
\end{remark}

\begin{notation} 
Let ${\mathbb{G}(d,V)\subset \mathbb{P}\bigl(\bigwedge^{d}\mathbb{V}\bigr)}$ be  the Grassmannian  of 
$d$-dimensional linear spaces in $V$.

If $\wv\in \bigwedge^dV$, we use the notation $[\wv]$ to indicate its projectivization in $ \mathbb{P}\bigl(\bigwedge^{d}\mathbb{V}\bigr)$.

For any element $[\wv]\in \mathbb{G}(d,V) $ we denote by $\vv$ the corresponding vector space of dimension $d$, i.e.\ $\vv\simeq \mathbb{C}^d\subset V$. 

Let ${t\in \bigwedge^dV}$ and denote by ${t^{\perp}\subset \bigwedge^{\bullet}V^*}$ its orthogonal via the product defined in \eqref{dot}, i.e.\ 
\[
t^{\perp}:= \bigl\{ \textstyle{\wh \in \bigwedge^{i\le n+1} V^*} \mid \wh \cdot t =0 \bigr\}.
\]

If ${[\wh]\in  \mathbb{G}(s,V^*)}$, then  $\vh^{\perp}$ is the  vector space of $V$ of dimension ${n+1-s}$ orthogonal to $\wh$.
\end{notation}

\begin{lemma}\label{inter}
Consider ${\wv = v_{1}\wedge \cdots \wedge v_{d}\in \bigwedge^d V}$ and ${\wh= h_{1}\wedge \cdots \wedge h_{s} \in \bigwedge^s V^{*}}$, with ${0\le s\le d}$, such that the subspaces  $\vh^{\perp}$ and $\vv$ intersect properly in $V$. Then the catalecticant satisfies ${\cC_{\wv}^{s,d-s}(\wh)=\wh\cdot \wv = u_{1}\wedge \cdots \wedge u_{d-s}}$ where ${u_{1},\ldots, u_{d-s}}$ is a basis of ${\vh^{\perp} \cap \vv}$. 
\end{lemma}
\begin{proof}
By multilinearity and skew-symmetry of relation \eqref{dot} in $h_{i}$ and $v_{i}$, we can assume that ${v_{1}=u_{1},\ldots, v_{d-s}=u_{d-s}}$ is a basis of  ${\vh^{\perp} \cap \vv}$. Since ${\wh(\wv_{R})=0}$ if ${R\cap \{1, \ldots, d-s\} \neq \emptyset}$, we have
\[
\cC_{\wv}^{s,d-s}(\wh)=\wh\cdot \wv = (-1)^{s(d-s)} \wh(\wv_{\{d-s+1, \ldots, d\}})\, u_{1}\wedge \cdots \wedge u_{d-s}.
\] 
As the intersection $\vh^{\perp} \cap \vv$  is proper, we have $\wh(\wv_{\{d-s+1, \ldots, d\}})\neq 0$, which proves the result replacing $u_{1}$ by $(-1)^{s(d-s)} \wh(\wv_{\{d-s+1, \ldots, d\}}) u_{1}$.
\end{proof}

\begin{lemma}\label{kerimg}
  For $\wv = v_{1}\wedge \cdots \wedge v_{d}\in \bigwedge^d V$, with ${0\le s\le d}$, 
\[ 
\ker \cC_{\wv}^{s,d-s}= (\vv^{\perp})_{s},\  \img \cC_{\wv}^{s,d-s}= {\textstyle \bigwedge^{d-s} \vv}.
\]
\end{lemma}
\begin{proof}
By Lemma \ref{inter}, for any linearly independent set of vectors  ${u_{1},\ldots, u_{d-s}\in \vv}$, one can find linearly independent hyperplanes ${h_{1},\ldots, h_{s}}$ such that
\[
{\cC_{\wv}^{s,d-s}(h_{1}\wedge \cdots \wedge h_{s})= u_{1}\wedge \cdots \wedge u_{d-s}}.
\]
This shows that the image of $\cC_{\wv}^{s,d-s}$ is $\bigwedge^{d-s} \vv$.

For any ${h_{1}\in \vv^{\perp}}$ and any ${h_{2}, \ldots, h_{s}\in V^{*}}$, an explicit computation shows that ${\wh=h_{1}\wedge \cdots \wedge h_{s}\in \bigwedge^{s} V^{*}}$ is such that ${\wh\cdot \wv=0}$. Therefore ${\vv^{\perp}\cap \bigwedge^{s} V^{*}\subset \ker \cC_{\wv}^{s,d-s}}$. The dimension of ${\vv^{\perp}\cap \bigwedge^{s} V^{*}}$ is 
${\dim \bigwedge^{s} V^{*} - \dim \bigwedge^{s} W^{*}}$ where ${W^{*}\oplus \vv^{\perp}=V^{*}}$, that is ${\binom{n+1}{s}-\binom{d}{s}}$. As the dimension of ${\img \cC_{\wv}^{s,d-s}= \bigwedge^{d-s} \vv}$ is ${\binom{d}{d-s}=\binom{d}{s}}$, we deduce that ${\ker \cC_{\wv}^{s,d-s}= \vv^{\perp}\cap \bigwedge^{s} V^{*}}$.
\end{proof}

We want now to prove the skew-symmetric analog of the Apolarity Lemma (Lemma \ref{ApolarityLemma}  c.f. \cite{IK,dol}).

\begin{remark} If ${t\in \bigwedge^dV}$, then ${t^{\perp}=\bigoplus_i \ker (\mathcal{C}_t^{i, d-i})}$. 
\end{remark}

Let us define the skew-symmetric analog of an ideal of points.
\begin{definition} 
Let ${\wv_{i}:=v_i^{(1)}\wedge \cdots \wedge v_i^{(d)} \in \bigwedge^dV}$ for ${i=1 , \ldots , r}$ be $r$ points. We define 
\begin{equation}
I^{\wedge}(\wv_{1}, \ldots , \wv_{r})= \bigcap_{i=1}^{r} \vv_{i}^{\perp}.
\end{equation}
\end{definition}
\begin{remark}\label{rem:ker}If $t= \sum_{i} \wv_{i}$ with $\wv_{i}\in \Gr{d}{V}$, then by Lemma \ref{kerimg} we have
\[
I^{\wedge}(\wv_{1}, \ldots , \wv_{r})_{s}= \bigcap_{i=1}^{r} (\vv_{i}^{\perp})_{s}\subset \ker \cC_{t}^{s,d-s}
\]
for ${s\le d}$. In particular, for ${s=d}$,
\[
\bigcap_{i=1}^{r} (\vv_{i}^{\perp})_{d} = \bigcap_{i=1}^{r}  \ker \cC_{\wv_{i}}^{d,0}=
  \bigl\{ \wh \in {\textstyle \bigwedge^{d} V^{*}}\mid  \wh(\wv_{i})=0, i=1,\ldots, r \bigr\}.
\]
\end{remark}

We can now prove the \emph{Skew-Apolarity Lemma}. The original formulation for the symmetric case is well described both in \cite{IK} and in \cite[Lemma 3.1]{dol}. It is worth noting that the only (crucial) difference between the two Lemmas is the product that in the classical case is the symmetric product while we prove that the same formulation holds also with the  skew-symmetric product. 

\begin{lemma}[Skew-apolarity Lemma] 
Let ${\wv_{i}=  v_i^{(1)}\wedge \cdots \wedge v_i^{(d)}\in \Gr{d}{V}\subset \bigwedge^{d} V}$ and let ${t\in \bigwedge^d V}$. The following are equivalent:
\begin{enumerate}
\item\label{i} The tensor $t$ can be written as ${t= \sum_{i=1}^r a_i \wv_{i}}$, with ${a_1,\ldots,a_r\in K}$;
\item\label{ii} $\bigcap _{i}(\vv_i^{\perp}) \subset (t^{\wedge \perp})$;
\item\label{iii} $\bigcap _{i}(\vv_i^{\perp})_{d}\subset (t^{\wedge \perp})_d$.
\end{enumerate}
\end{lemma}

\begin{proof} 
The fact that (\ref{i}) implies (\ref{ii}) and that (\ref{ii}) implies (\ref{iii}) is obvious.
Let us prove that (\ref{iii}) implies (\ref{i}). Observe that any non-zero element $\wh\in I^{\wedge}(\wv_{1}, \ldots, \wv_r)_d\subset \bigwedge^{d}V^{*}$ can be seen as a hyperplane in $\bigwedge^d V$. By Remark \ref{rem:ker},  $I^{\wedge}(\wv_{1}, \ldots, \wv_r)_d$ is the set of hyperplanes which contain the points $\wv_i\in \Gr{d}{V}$ for $i=1, \ldots ,r$. Now, condition (\ref{iii}) is equivalent to saying that any such $\wh$ contains the point $t$. Therefore, up to scalar multiplications, $t$ can be written as a linear combination of the $\wv_j$'s for $j=1, \ldots , r$.
\end{proof}

Comparing the statements of the Skew-apolarity Lemma and the Apolarity Lemma, we see the similarities, if we observe that $\bigcap _{i}(\vv_i^{\perp})$ plays the role of the ideal of the scheme  ${X=\{ [l_1^*], \ldots , [l_r^*]\}}$.

\begin{remark}\label{essentialvars}
Like in the symmetric case, one can define \emph{essential variables}  for an element $t\in\bigwedge^dV$ to be a basis of the smallest vector subspace $W\subseteq V$ such that $t\in\bigwedge^dW$ (it is a classical concept but for modern references see \cite{IK,c}). We can check this by computing the kernel of the first catalecticant $\cC_{t}^{1,d-1}$.
\end{remark}

\section{A few examples of ideals of points}\label{a:few:ex}

Once the skew-symmetric apolar ideal has been defined, one could be interested in studying the analogous of the Hilbert function.
A first obvious observation is that the situation is intrinsically very different from the symmetric case, where any ideal has elements in any degree greater or equal to the degree of the smallest generator, while in the skew-symmetric case we won't have tensors in degree higher than the dimension of ${V}$.
It is not the purpose of this paper to give an exhaustive description of the Hilbert function of zero-dimensional schemes in the skew-symmetric situation, but we would like to present some first examples in the cases of ideal of points. 

A first example where things are very different from the symmetric case is the following. In the symmetric case the Hilbert function of $r$ generic points is $r$ for any degree  ${d\geq r-1}$, while, as we are going to see in the next Lemma, if ${rd\le n+1=\dim V}$ the ideal of $r$ points in the skew-symmetric case is generated in degrees 1 and 2 only.

\begin{lemma}\label{lemma:few:general:points}
Let $d$ and $r$ be positive integers such that ${rd\le n+1=\dim V}$, and let $[\wv_1], \ldots , [\wv_r] \in\Gr{d}{V}$ be $r$ general points. Then $I^{\wedge}(\wv_1,  \ldots, \wv_r)$ is generated in degrees $1$ and $2$. {In fact, with these hypotheses, generators of degree 1 occur only if $rd<n+1$, while if $rd=n+1$ then $I^{\wedge}(\wv_1, \ldots , \wv_r)$ is generated only in degree 2.}
\end{lemma}
\begin{proof}
Since ${\wv_1,  \ldots, \wv_r}$ are general elements, we can assume that $V$ admits a basis ${e_1,\ldots,e_{n+1}}$ such that
\[
\wv_1=e_1 \wedge \cdots \wedge e_d,\ \wv_2=e_{d+1} \wedge \cdots \wedge e_{2d},\ \ldots,\  \wv_r=e_{(r-1)d+1} \wedge \cdots \wedge e_{rd}.
\]
Let $\{e_1^*,\ldots,e_{n+1}^*\}\subset V^*$ be the dual basis. Then we can check that $I^{\wedge}(\wv_1,  \ldots, \wv_r)$ is generated by $e_{rd+1}^*,\ldots,e_{n+1}^*$, in degree $1$, and degree-two elements $e_i^*\wedge e_j^*$ such that ${1\le i\le sd<j\le rd}$, for some ${s<r}$.  It is clear that these elements are in the ideal. To see that they are enough to generate it, take ${\wh\in\bigwedge^bV^*}$ and write ${\wh=\sum_{1\le i_1 < \cdots < i_b\le n+1}a_{i_1\cdots i_b}e_{i_1}^* \wedge \cdots \wedge e_{i_b}^*}$. Then the coefficients of ${\wh \cdot \wv_j}$ in the basis ${\{e_{i_1}\wedge \cdots \wedge e_{i_{d-b}}\}_{1\le i_1< \cdots < i_{d-b} \le n+1}}$ of $\bigwedge^{d-b}V$ are the elements $a_{i_1\cdots i_b}$, with ${(j-1)d < i_1< \cdots < i_b \le jd}$. 
So $\wh$ will be in the ideal only if all such ${a_{i_1\cdots i_b}}$ vanish. 
Hence in every non-zero term in $\wh$ at least one of the generators mentioned above occurs.  
\end{proof}

\begin{example}\label{twointersecting}
Let $d$ be an integer such that ${2d\ge n+1}$, and let $[\wv_1],[ \wv_2] \in {\mathbb{G}({d},{V})}$ be two general points. Then $I^{\wedge}(\wv_1,  \wv_2)$ is generated in degree $2$ and,  if $d+1<n$, in degree ${n+1-d}$. Since $\wv_1$ and $\wv_2 $ are general, they represent $d$-dimensional subspaces ${\vv_1, \vv_2}$ of $V$ that intersect minimally in dimension ${2d-(n+1)}$, which is assumed to be non-negative. Choose a basis ${e_1,\ldots,e_{n+1}}$ for $V$ such that
\begin{align*}
\wv_1&=e_1 \wedge \cdots \wedge e_{2d-n-1} \wedge e_{2d-n} \wedge \cdots \wedge e_d,\\
\wv_2&=e_1 \wedge \cdots \wedge e_{2d-n-1} \wedge e_{d+1} \wedge \cdots \wedge e_{n+1}.
\end{align*}
Let $\{e_1^*,\ldots,e_{n+1}^*\}$ be the dual basis. If ${d\neq n-1}$ then $I^{\wedge}(\wv_1,  \wv_2)_2$ is generated by elements $e_i^*\wedge e_j^*$ such that ${2d-n\le i\le d<j\le n+1}$, while if ${n\geq 3}$ and ${d=n-1}$ then, among the generators of ${I^{\wedge}(\wv_1,  \wv_2)_2}$, apart from the previous ones, there is also ${e_{n-2}^*\wedge e_{n-1}^*-e_{n}^*\wedge e_{n+1}^*}$; moreover, in both cases, ${I^{\wedge}(\wv_1,  \wv_2)_{d-(2d-n)-1}}$ is generated by ${e_{2d-n}^*\wedge \cdots \wedge e_d^*-e_{d+1}^*\wedge \cdots \wedge e_{n+1}^*}$.

Remark that if ${2d=n+1}$ then
\begin{align*}
\wv_1&=e_1\wedge \cdots \wedge e_d,\\
\wv_2&=e_{d+1}\wedge \cdots \wedge e_{2d}
\end{align*}
and we are in the case of Lemma \ref{lemma:few:general:points} where we see that $I^{\wedge}(\wv_1,\wv_2)$ is generated only in degree 2 by $e_i^*\wedge e_j^*$, with $1\leq i \leq d < j\leq n+1$.
\end{example}

\begin{example}\label{threeintersecting}
Let $d$ be an integer such that ${3d\ge 2(n+1)}$, and let $[\wv_1]$, $[\wv_2]$, and $[\wv_3]$ be three points in $\mathbb{G}({d},{V})$ such that if we choose a basis ${e_1,\ldots,e_{n+1}}$ for $V$ they can be represented as
\begin{align*}
\wv_1&=e_1 \wedge \cdots \wedge e_{3d-2(n+1)} \wedge e_{3d-2(n+1)+1}\wedge\cdots \wedge 
  e_{2d-n-1} \wedge e_{2d-n} \wedge \cdots \wedge e_d,\\
\wv_2&=e_1 \wedge \cdots \wedge e_{3d-2(n+1)} \wedge e_{3d-2(n+1)+1}\wedge\cdots \wedge 
  e_{2d-n-1} \wedge e_{d+1} \wedge \cdots \wedge e_{n+1},\\
\wv_3&=e_1 \wedge \cdots \wedge e_{3d-2(n+1)}  \wedge e_{2d-n} \wedge\cdots \wedge e_d \wedge e_{d+1} \wedge \cdots \wedge e_{n+1}.
\end{align*}
Let $\{e_1^*,\ldots,e_{n+1}^*\}$ be the dual basis. Then $I^{\wedge}(\wv_1,  \wv_2)$ is generated by elements ${e_i^*\wedge e_j^*\wedge e_k^*}$ such that
${3d-2(n+1)< i\le 2d-n-1< j\le d<k\le n+1}$.   
\end{example}
\red{
%
%
}

\begin{example} 
Let ${n=3}$ and ${d=2}$. Let $\{e_0, e_1, e_2, e_3\}$ be a basis for $V$ and let $\{e_0^*, e_1^*, e_2^*, e_3^*\}$ be the dual basis. Consider the following vectors in $\bigwedge^2V$:
\begin{align*}
\wv_1&=e_0 \wedge e_1,\\
\wv_2&=e_2 \wedge e_3,\\
\wv_3&=(e_0+e_2) \wedge (e_1+e_3) = e_0 \wedge e_1 + e_0 \wedge e_3 - 
  e_1 \wedge e_2 + e_2 \wedge e_3,\\
\wv_4&=(e_0+e_3) \wedge (e_1+e_2) = e_0 \wedge e_1 + e_0 \wedge e_2 - 
  e_1 \wedge e_3 - e_2 \wedge e_3.
\end{align*}
As we saw in Example \ref{twointersecting}, $I^{\wedge}(\wv_1, \wv_2)$ is generated in degree two. It is easy to verify that also $I^{\wedge}(\wv_1, \wv_2, \wv_3)$ and $I^{\wedge}(\wv_1, \wv_2, \wv_3, \wv_4)$ are generated in degree two, and 
\begin{align*}
I^{\wedge}(\wv_1, \wv_2) &= ( e_0^* \wedge e_2^*,\ e_0^* \wedge e_3^*,\ 
  e_1^* \wedge e_2^*,\ e_1^* \wedge e_3^* ),\\
I^{\wedge}(\wv_1, \wv_2, \wv_3) &= ( e_0^* \wedge e_2^*,\ e_0^* \wedge e_3^* + 
  e_1^* \wedge e_2^*,\ e_1^* \wedge e_3^* ),\\
I^{\wedge}(\wv_1, \wv_2, \wv_3, \wv_4) &= ( e_0^* \wedge e_2^* + e_1^* \wedge e_3^*,\ 
  e_0^* \wedge e_3^* + e_1^* \wedge e_2^* ).
\end{align*}
\end{example}

\begin{example}
Let ${n=5}$ and ${d=4}$. Let $\{e_0, \ldots, e_5\}$ be a basis for $V$ and let $\{e_0^*, \ldots, e_5^*\}$ be the dual basis. Consider the following vectors in $\bigwedge^4V$:
\begin{align*}
\wv_1&=e_0 \wedge e_1 \wedge e_2 \wedge e_3,\\
\wv_2&=e_0 \wedge e_1 \wedge e_4 \wedge e_5,\\
\wv_3&=e_2 \wedge e_3 \wedge e_4 \wedge e_5,\\
\wv_4&=(e_0+e_2) \wedge (e_1+e_3) \wedge (e_0+e_4) \wedge (e_1+e_5).
\end{align*}
Again as in Examples \ref{twointersecting} and \ref{threeintersecting}, the ideals $I^{\wedge}(\wv_1, \wv_2)$ and $I^{\wedge}(\wv_1, \wv_2, \wv_3)$ are generated in degrees $2$ and $3$, respectively. However, $I^{\wedge}(\wv_1, \wv_2, \wv_3, \wv_4)$ has two generators in degree $3$ and five in degree $4$. Here are the generating sets for each ideal:
\begin{align*}
I^{\wedge}(\wv_1, \wv_2) &= ( e_2^* \wedge e_4^*,\ e_2^* \wedge e_5^*,\ 
  e_3^* \wedge e_4^*,\ e_3^* \wedge e_5^* ),\\
I^{\wedge}(\wv_1, \wv_2, \wv_3) &= ( e_0^* \wedge e_2^* \wedge e_4^*,\ 
    e_0^* \wedge e_2^* \wedge e_5^*,\ e_0^* \wedge e_3^* \wedge e_4^*,\ 
    e_0^* \wedge e_3^* \wedge e_5^*,\\ 
  &\qquad e_1^* \wedge e_2^* \wedge e_4^*,\ 
    e_1^* \wedge e_2^* \wedge e_5^*,\ e_1^* \wedge e_3^* \wedge e_4^*,\ 
    e_1^* \wedge e_3^* \wedge e_5^* ),\\
I^{\wedge}(\wv_1, \wv_2, \wv_3, \wv_4) &= \big( e_0^* \wedge e_2^* \wedge e_4^*,\ 
    e_1^* \wedge e_3^* \wedge e_5^*,\ 
    e_0^* \wedge e_1^* \wedge e_2^* \wedge (e_3^*+e_5^*),\\ 
  &\qquad e_0^* \wedge e_1^* \wedge e_2^* \wedge e_5^* + 
    e_0^* \wedge e_1^* \wedge e_3^* \wedge e_4^*,\ 
    e_0^* \wedge e_1^* \wedge e_4^* \wedge (e_3^*+e_5^*),\\ 
  &\qquad e_0^* \wedge e_1^* \wedge e_4^* \wedge e_5^* - 
    e_0^* \wedge e_2^* \wedge e_3^* \wedge e_5^*,\ 
    e_0^* \wedge (e_2^*+e_4^*) \wedge e_3^* \wedge e_5^*,\\ 
  &\qquad e_0^* \wedge e_3^* \wedge e_4^* \wedge e_5^* + 
    e_1^* \wedge e_2^* \wedge e_3^* \wedge e_4^*,\ 
    e_0^* \wedge e_2^* \wedge (e_3^*+e_5^*) \wedge e_4^*,\\ 
  &\qquad (e_1^*+e_3^*) \wedge e_2^* \wedge e_4^* \wedge e_5^* 
    \big).
\end{align*}

\end{example}

\section{On conditions for the rank of a skew-symmetric tensor}\label{3vector}

The following example highlights another very big difference with the symmetric case: in the symmetric case if we have an element of type ${[L_1]^d+[L_2]^d}$, with $[L_1]$ and $[L_2]$ distinct points in ${\mathbb{P}(V)}$, this has symmetric rank $2$ for all ${d>1}$ since Veronese varieties are cut out by quadrics but do not contain lines. However, the next Lemma shows that the same is not going to happen in the skew-symmetric case.

\begin{lemma}\label{line:in:Grass}
Let ${[\wv_1],  [\wv_2 ]\in {\mathbb{G}({d},{V})\subset \mathbb{P}\bigl(\bigwedge^d,V\bigr)}}$ be two distinct points. The tensor ${\wv=\wv_1 +  \wv_2}$ has rank one if and only if the line passing through $\wv_1$ and $\wv_2$ is contained in ${\mathbb{G}({d},{V})}$.
\end{lemma}

\begin{proof}
Since the Grassmannian ${\mathbb{G}({d},{V})}$ is cut out by quadrics, if the line $\ell$ through $[\wv_1]$ and $[\wv_2]$ also meets the Grassmannian in ${[\wv_1 +  \wv_2]}$ then $\ell$ must be contained in ${\mathbb{G}({d},{V})}$.
\end{proof}

The following lemma is probably a classically known fact.

\begin{lemma}\label{dimensionrmk}
Let ${[\wv_1],  [\wv_2 ]\in {\mathbb{G}({d},{V})\subset\mathbb{P}\bigl(\bigwedge^dV\bigr)}}$. The tensor ${\wv=\wv_1 +  \wv_2}$ has skew\-\mbox{-symmetric} rank $1$ if and only if the intersection of the subspaces  $\vv_1$ and $\vv_2$ has dimension at least ${d-1}$. 
\end{lemma}
\begin{proof}
First of all remark that if ${\dim(\vv_1\cap \vv_2)=d}$ then $\wv_1$ and $\wv_2$ represent the same space, so there exist linearly independent vectors ${v_1, \ldots , v_d\in V}$ such that ${\wv_i=\alpha_{i,1}v_1\wedge \cdots \wedge \alpha_{i,d}v_d}$, $i=1,2$, then clearly
\[
{\wv=\wv_1+\wv_2=(\alpha_{1,1}+\alpha_{2,1})v_1\wedge \cdots \wedge (\alpha_{1,d}+\alpha_{2,d})v_d}.
\] 
In the case ${\dim (\vv_1\cap\vv_2)=d-1}$ there exists a subspace ${\vec{\mathbf{w}}\subset V}$ of dimension ${d-1}$ such that ${\wv_i=\mathbf{w}\wedge v_{i}}$ for ${i=1,2}$. Then ${\wv_1+\wv_2=\mathbf{w}\wedge(v_1+v_2)}$, so if $\{w_1, \ldots , w_{d-1}\}$ is a basis for $\vec{\mathbf{w}}$ we have that ${\wv=\wv_1+\wv_2=w_1\wedge \cdots \wedge w_{d-1}\wedge(v_1+v_2)}$, which has rank $1$. 

Conversely, assume that  ${\dim  (\vv_1\cap\vv_2)=k\leq d-2}$. We can prove that ${\wv=\wv_1+\wv_2}$ doesn't have rank $1$ by induction on $d$. If ${d=2}$ then ${\wv=v_1\wedge v_2 +v_3\wedge v_4}$ with ${v_1 , \ldots , v_4}$ linearly independent in $V$. Such a $\wv$ has skew\--symmetric rank 1 as skew\-\mbox{-symmetric} tensor if and only if the skew\--symmetric matrix which it represents has rank $2$, which is impossible since  ${v_1 , \ldots , v_4}$ are linearly independent in $V$ hence $\wv$ has rank $4$ as a matrix. Now if ${d>2}$ we have that there exist ${w_1, \ldots , w_k\in V}$ linearly independent vectors such that $\wv_i=w_1\wedge \cdots \wedge w_k \wedge v_{i,k+1}\wedge \cdots \wedge v_{i,d}$ with  $v_{i,j}$ linearly independent for $i=1,2$ and $j=k+1, \ldots , d$; so $\wv=w_1\wedge \cdots \wedge w_k \wedge(v_{1,k+1}\wedge \cdots \wedge v_{1,d} +  v_{2, k+1} \wedge \cdots \wedge v_{2,d})$ which has rank one if and only if $\tilde \wv=v_{1,k+1}\wedge \cdots \wedge v_{1,d} +  v_{2, k+1} \wedge \cdots \wedge v_{2,d}\in \bigwedge^{d-k}V$ has rank 1. If $k=d-1$ then $\tilde\wv $ is a matrix that by the same reason as above doesn't have (as a matrix) rank 2 hence $\wv$ as a tensor doesn't have skew symmetric rank 1. If $k<d-2$ one can again argue by induction and easily conclude.
\end{proof}

\begin{remark}\label{UwedgeWremark} 
Let $U$ and $W$ be subspaces of $V$ such that ${U\cap W=\{0\}}$. Then both ${\PP(U\wedge W)}$ and $\mathbb{G}(1,V)$ are subvarieties of $\PP\bigl(\bigwedge^2V\bigr)$, and we wish to see what their intersection looks like. Since ${U\wedge W\subset(U+W) \wedge (U+W)}$, any rank-one tensor in ${U\wedge W}$ can be written as ${(u_1+w_1) \wedge (u_2+w_2)}$, with ${u_1,u_2\in U}$ and ${w_1,w_2\in W}$. However, when we expand this, we get
\[
(u_1+w_1) \wedge (u_2+w_2) = u_1 \wedge u_2 + u_1 \wedge w_2 + w_1 \wedge u_2 + w_1 \wedge w_2, 
\] 
so both $u_1 \wedge u_2$ and $w_1 \wedge w_2$ must vanish. But then the tensor ${u_1 \wedge w_2+w_1 \wedge u_2}$ has rank $1$, so by Remark \ref{dimensionrmk} the elements ${u_1 \wedge w_2}$ and ${w_1 \wedge u_2}$ correspond to lines in $\PP(V)$ with one point in common. Therefore there is ${v\in V}$ and scalars ${a_1,a_2,b_1,b_2\in K}$ such that ${a_1u_1+b_2w_2=v=a_2u_2+b_1w_1}$.  However, this means that ${a_1u_1-a_2u_2=b_1w_1-b_2w_2}$, and since ${U\cap W=\{0\}}$ we conclude that at least one of the sets $\{u_1,u_2\}$ or $\{w_1,w_2\}$ is linearly dependent, so  ${u_1 \wedge w_2 + w_1 \wedge u_2=\tilde u\wedge \tilde w}$ for some ${\tilde u\in U}$ and ${\tilde w\in  W}$. Therefore the intersection ${\PP(U\wedge W)\cap\mathbb{G}(1,E)}$ is isomorphic to the Segre variety $\PP( U\otimes  W)$. 

Observe that if we drop the condition ${ U\cap W=\{0\}}$, this is no longer the case: let ${U=\langle e_0,e_1\rangle}$ and ${W=\langle e_0,e_2\rangle}$, where ${e_0,e_1,e_2}$ are independent. Then
\[{e_0\wedge e_2 + e_1 \wedge e_0= e_0 \wedge (e_1-e_2)},
\]
but we cannot write this tensor as ${\tilde u\wedge \tilde w}$, with ${\tilde u\in U}$ and ${\tilde w\in W}$.
\end{remark}

\section{The trivector case}
\label{sec:trivector}
In this section we study in detail the situation of vectors in $\bigwedge^3\mathbb{C}^n$ for $n\leq 8 $, we call these elements ``trivectors'' in agreement with the notation used by Gurevich in his book \cite{gur1}  since we will extensively use his characterization of normal forms.

Let us start with an example which was already well known to C. Segre (\cite[Paragraph 28]{Segre}).

\begin{example}[Segre]\label{triky}
Let ${n=5}$, let $\{f_0,\ldots,f_5\}$ be a basis of $V$, and let $\{f_0^*,\ldots,f_5^*\}$ be the dual basis. Consider the vectors  
\begin{align*}
\wv_1&=f_0 \wedge f_1 \wedge f_2, & \wv_2&=f_0 \wedge f_3 \wedge f_4,
  &\text{and } \wv_3&=f_1 \wedge f_3 \wedge f_5,
\end{align*}
And let ${\wv=\wv_1+\wv_2+\wv_3}$. Then ${\ker\mathcal{C}_\wv^{1,2}=I^{\wedge}(\wv_1, \wv_2, \wv_3)_1=0}$, and
\begin{equation}
\label{I2divisible}
\begin{aligned}
\ker\mathcal{C}_\wv^{2,1}&=\langle f_0^* \wedge f_2^* + f_3^* \wedge f_5^*,\ 
    f_0^* \wedge f_4^* - f_1^* \wedge f_5^*,\ f_0^* \wedge f_5^*,\ f_1^* \wedge f_2^* - f_3^* \wedge f_4^*,\\  
  &\qquad f_1^* \wedge f_4^*,\ 
    f_2^* \wedge f_3^*,\ f_2^* \wedge f_4^*,\ f_2^* \wedge f_5^*,\ f_4^* \wedge f_5^* \rangle.
\end{aligned}
\end{equation}
We claim that $\wv$ has rank $3$. Suppose that this is not the case, and write ${\wv=\wv_4+\wv_5}$, where ${\wv_4=g_0 \wedge g_1 \wedge g_2}$ and ${\wv_5=g_3 \wedge g_4 \wedge g_5}$. Since ${\ker\mathcal{C}_\wv^{1,2}=0}$, we must have ${V=\langle g_0,\ldots,g_5\rangle}$. So ${g_0,\ldots,g_5}$ are independent, and therefore
\[
I^{\wedge}(\wv_4, \wv_5)_2=({\wv_4}^{\perp})_1 \wedge ({\wv_5}^{\perp})_1
  =\langle g_0^*, g_1^*, g_2^* \rangle \wedge \langle g_3^*, g_4^*, g_5^* \rangle .
\]
Since ${I^{\wedge}(\wv_4, \wv_5)_2\subseteq\ker\mathcal{C}_\wv^{2,1}}$ and both spaces have dimension $9$, equality must hold.

Then, we have 
\begin{equation}\label{blue}
I^{\wedge}(\wv_4, \wv_5)_{3}= \left\langle
\begin{array}{l}
f^{*}_{0}\wedge f^{*}_1\wedge f^{*}_{4},
f^{*}_{0}\wedge f^{*}_1\wedge f^{*}_{5},
f^{*}_{0}\wedge f^{*}_2\wedge f^{*}_{3},
f^{*}_{0}\wedge f^{*}_2\wedge f^{*}_{4},\\
f^{*}_{0}\wedge f^{*}_2\wedge f^{*}_{5},
f^{*}_{0}\wedge f^{*}_3\wedge f^{*}_{5},
f^{*}_{0}\wedge f^{*}_4\wedge f^{*}_{5},
f^{*}_{2}\wedge f^{*}_{3}\wedge f^{*}_{4},\\
f^{*}_{2}\wedge f^{*}_{3}\wedge f^{*}_{5},
f^{*}_{2}\wedge f^{*}_{4}\wedge f^{*}_{6},
f^{*}_{3}\wedge f^{*}_{4}\wedge f^{*}_{5},\\
f^{*}_{0}\wedge f^{*}_{1}\wedge f^{*}_{2} - f^{*}_{1}\wedge f^{*}_{3}\wedge f^{*}_{5},
  f^{*}_{0}\wedge f^{*}_{3}\wedge f^{*}_{4} - f^{*}_{1}\wedge f^{*}_{3}\wedge f^{*}_{5}
\end{array}
\right\rangle
\end{equation}
and 
\[ 
I^{\wedge}(\wv_4, \wv_5)_{3}^{\perp}= \langle f_{0}\wedge f_1\wedge f_2 + f_{0}\wedge f_3\wedge f_4 + f_{1}\wedge f_3\wedge f_5,f_{0}\wedge f_1\wedge f_3\rangle.
\]
But, if  $\wv=\wv_{4}+\wv_{5}$ and $\vv_{4}\cap \vv_{5}=\{0\}$ then $I^{\wedge}(\wv_{4},\wv_{5})_{3}^{\perp}= \langle \wv_{4}, \wv_{5} \rangle$. This implies that ${\PP(I^{\wedge}(\wv_{4},\wv_{5})_{3}^{\perp} )\cap \Gr{3}{V}=\{\wv_{4},\wv_{5}\}}$.

An explicit computation shows that 
\[
\lambda (f_{0}\wedge f_1\wedge f_2 + f_{0}\wedge f_3\wedge f_4 
  + f_{1}\wedge f_3\wedge f_5)+ \mu f_{0}\wedge f_1\wedge f_3 \in \Gr{3}{V}
\] with ${\lambda,\mu\in K}$ implies that ${\lambda=0}$ (we use the Pl\"ucker relation $[0,1,5][1,2,3]-[0,1,2][1,3,5]+[0,1,3][1,2,5]=0= -\lambda^{2}$). This contradicts the property that $\PP(I^{\wedge}(\wv_{4},\wv_{5})_{3}^{\perp}) \cap \Gr{3}{V}=\{\wv_{4},\wv_{5}\}$. Therefore,  $\wv$ must be of rank $3$.
\end{example}
  
Note that if $\wv_4$ and $\wv_5$ are as above, ${\dim\ker\mathcal{C}_\wv^{s,3-s}=\dim\ker\mathcal{C}_{\wv_4+\wv_5}^{s,3-s}}$, for any degree $s$. Both kernels in degree two intersect ${\Gr{3}{V}}$ in (projective) dimension $4$. So we cannot tell the rank of a tensor from computing these dimensions. In this case, it is the structure of ${\ker\mathcal{C}_\wv^{3,0}}$ that allowed us to show that $\wv$ has rank $3$. 

Note also that $\wv$ is the normal form of a vector not belonging to the orbit closure of a rank 2 skew-symmetric tensor. Therefore, since $	\sigma_3(\mathbb{G}(3,V))$ fills the ambient space, this shows that $\wv$ has rank $3$  (cf. \cite[\S 35.2, case IV]{gur1}).

\bigskip
 
If we use the classification given by Gurevich in his book \cite[Chapter VII]{gur1} of the normal forms of the trivectors, i.e.\ skew-symmetric tensors in $\bigwedge^3 \mathbb{C}^{n+1}$, we deduce the following description. 

\subsection{Trivectors in  $\mathbb{P}^2$ or $\mathbb{P}^3$} 
If ${n=2,3}$ there is only one possibility for a projective class of a trivector that is to be of skew-symmetric rank 1, i.e.\ 
\begin{equation}\tag{II}\label{II}[\wv]=[v_0 \wedge v_1 \wedge v_2].\end{equation}
If ${n=3}$ then $I^{\wedge}(\wv)$ is generated in degree $1$ by $I^{\wedge}(\wv)_1$ as in Lemma \ref{lemma:few:general:points}, in particular ${I^{\wedge}(\wv)= (v_3^*)}$ where ${\langle v_3\rangle=\langle v_0,v_1,v_2 \rangle^\perp \subset V}$.  

Therefore, assume that ${\wv\in \bigwedge^3 \mathbb{C}^4}$ is given, if one wants to find its decomposition as in (\ref{II}), one has simply to compute a basis  $\{v_0,v_1,v_2\}$  of $I^{\wedge}(\wv)_1^{\perp}$, and such a basis will be good for the presentation of $\wv$ as a tensor of skew-symmetric rank 1 as in (\ref{II}).

\subsection{Trivectors in  $\mathbb{P}^4$} 
If ${n=4}$ there is one more possibility for a trivector with respect to the previous case (\ref{II}),  that is 
\begin{equation}\tag{III}\label{III}
[v_0 \wedge v_1 \wedge v_2+ v_0\wedge v_3 \wedge v_4].
\end{equation} 
In fact if ${n=4}$ any trivector ${\wv\in \bigwedge^3\mathbb{C}^5}$ is divisible by some vector, say ${\wv=v_0\wedge \wv}'$  where ${\wv'\in \bigwedge^2\mathbb{C}^5}$. Therefore if ${n=4}$ there are only two possibilities for the projective class of a tensor ${\wv\in \bigwedge^3\mathbb{C}^5}$: either it is of skew-symmetric rank 1 and it can be written as (\ref{II}) or it is of skew-symmetric rank 2 and it can be written as (\ref{III}).

This is a particular case of Example \ref{twointersecting} and if we call ${\wv_1=v_0\wedge v_1 \wedge v_2}$ and ${\wv_2=v_0\wedge v_3 \wedge v_4}$ we easily see that ${I^{\wedge}(\wv_1,\wv_2)}$ is generated in degree $2$ by 
\begin{equation}\label{pallino}
(v_1^*\wedge v_3^*,\, v_1^*\wedge v_4^*,\, v_2^*\wedge v_3^*,\, v_2^*\wedge v_4^*,\, 
  v_2^*\wedge v_4^*-v_1^*\wedge v_2^*).
\end{equation}

Notice that Lemma \ref{line:in:Grass} is confirmed: if we take ${\{v_0 , \ldots, v_4\}}$ to be a basis of $\mathbb{C}^5$ and we do the standard Pl\"uker embedding ${v_0\wedge v_1 \wedge v_2 \mapsto p_{0,1,2}}$, ${v_0\wedge v_3 \wedge v_4 \mapsto p_{0,3,4}}$ by a simple computation we get that the line through $p_{0,1,2}$ and $p_{0,3,4}$ is not contained in $\mathbb{G}({3},{\mathbb{C}^5})\subset \mathbb{P}\bigl(\bigwedge^3\mathbb{C}^5\bigr)$.

In any case, if ${n=4}$, in order to understand if a given tensor ${\wv\in\bigwedge^3\mathbb{C}^5}$ has rank $1$ or $2$, it is sufficient to compute   ${\dim \ker \mathcal{C}^{1,2}_\wv}$. If it is non trivial it means that there are generators of degree 1 and that we are in case (\ref{II}) of skew-symmetric rank 1; if ${\dim \ker \mathcal{C}^{1,2}_\wv=0}$ then we only have generators in degree $2$ and we are in case (\ref{III}) of skew-symmetric rank $2$.

Now assume that we want to see the skew-symmetric decomposition of a given tensor $\wv\in\bigwedge^3\mathbb{C}^5$. If we find generators of $I^{\wedge}(\wv)$ in degree 1, say $\{v_4^*, v_5^*\}$, then $\wv$ is of the form ${\wv=v_0\wedge v_1 \wedge v_2}$ where ${\langle v_0,v_1,v_2 \rangle = (v_4^*\wedge v_5^*)^\perp}$. If we do not find any generator in degree 1 and we want to recover the decomposition of $\wv=\wv_1+\wv_2$ as a skew-symmetric rank 2 tensor, we have to look at the structure of $I^{\wedge}(\wv_1,\wv_2)$ in degree 2. Notice that $I^{\wedge}(\wv_1,\wv_2)_2$ is exactly of the same structure of (\ref{pallino}). With this notation we have that $\wv_1=v_0\wedge v_1\wedge v_2$ and $\wv_2=v_0\wedge v_3 \wedge v_4$ where $v_i=(v_i^*)^*$ for $i=1, \ldots , 4$ and $\langle v_0\rangle =\langle v_1, v_2 ,v_3, v_4 \rangle^{\perp}$.

\subsection{Trivectors in  $\mathbb{P}^5$} 
If $n=5$ there are two more possibilities in addition to (\ref{II}) and (\ref{III}) for the normal form of the projective class of trivectors $\wv\in \bigwedge^3\mathbb{C}^6$:
\begin{equation}\tag{IV}\label{IV} [v_0\wedge v_1 \wedge v_2+v_0\wedge v_3 \wedge v_4+v_1\wedge v_3 \wedge v_5]
\end{equation}
and 
\begin{equation}\tag{V}\label{V} [v_0\wedge v_1 \wedge v_2+v_3\wedge v_4 \wedge v_5].
\end{equation}
Obviously (\ref{V}) corresponds to a tensor of skew-symmtetric rank $2$. Since (\ref{IV}) is in a different orbit with respect to all the others, it can be neither of skew-symmetric rank $1$ nor $2$. Therefore the presentation we have as sum of $3$ summands is minimal hence (\ref{IV}) has skew-symmetric rank $3$.

Notice that (\ref{V}) is the case of Lemma \ref{lemma:few:general:points} with ${d=3}$, ${r=2}$ and ${6=n+1=rd}$ where $I^{\wedge}(v_0\wedge v_1\wedge v_2, v_3\wedge v_4\wedge v_5)$ is generated in degree $2$ by ${v_i^*\wedge v_j^*}$ with ${0\leq i \leq 2 < j \leq 5}$. 

Case (\ref{IV}) is well described by Example \ref{triky}. Again, in this last case, we do not have generators in degree $1$ for ${I^{\wedge}(v_0\wedge v_1 \wedge v_2,v_0\wedge v_3 \wedge v_4,v_1\wedge v_3 \wedge v_5)}$, while the generators in degree $2$ are described by (\ref{blue}). 	This leads to the following algorithm for the skew-symmetric tensor decomposition of a tensor ${\wv\in \bigwedge^3 \mathbb{C}^6}$:

\begin{breakablealgorithm}\renewcommand{\thealgorithm}{1}\caption{: Algorithm for the skew-symmetric rank and a decomposition of  an element  in $ \bigwedge^3 \mathbb{C}^6$. }
\leftline{INPUT: $\wv\in \bigwedge^3 \mathbb{C}^6$.}
\leftline{OUTPUT: Decomposition and skew-symmetric rank of $\wv$}
   \begin{algorithmic}[1]
\STATE Compute $\ker \mathcal{C}_{\wv}^{1,2}$;
\IF {$\dim \ker \mathcal{C}_{\wv}^{1,2}=3$}
\STATE{$\wv$ has skew-symmetric rank 1 \AND   $$\wv=v_0\wedge v_1 \wedge v_2$$ where $\langle v_0,v_1,v_2\rangle =(\ker \mathcal{C}_{\wv}^{1,2})^{\perp}$.} 
\ELSE\STATE{Go to Step \ref{alg:3:6:item:3}}
\ENDIF
\STATE\label{alg:3:6:item:3} Compute $\ker \mathcal{C}_{\wv}^{2,1}$; 
\IF {$\dim \ker \mathcal{C}_{\wv}^{1,2}=1$}
\STATE{$\wv$ has skew-symmetric rank 2 \AND   $$\wv=v_0\wedge v_1 \wedge v_2+ v_0 \wedge v_3 \wedge v_4$$ where $\langle v_0,v_1,v_2,v_3,v_4\rangle =(\ker \mathcal{C}_{\wv}^{1,2})^{\perp}=\langle v_5^*\rangle$ \AND 
$\ker \mathcal{C}_{\wv}^{2,1}= \langle  v_1^*\wedge v_3^*, v_1^*\wedge v_4^*, v_2^*\wedge v_3^*, v_2^*\wedge v_4^*, v_2^*v_4^*-v_1^*v_2^*\rangle$}
\ELSE\STATE{$\dim \ker \mathcal{C}_{\wv}^{1,2}=0$ \AND go sto Step \ref{step:rank3};}
\ENDIF
\STATE\label{step:rank3}{$\dim \ker \mathcal{C}_{\wv}^{2,1}=9$ \AND}
\IF {there exist $v_0, \ldots , v_5\in \mathbb{C}^6$ linearly independent vectors such that $\ker \mathcal{C}_{\wv}^{2,1}=\langle v_i^*\wedge v_j^*\rangle_{i\in \{0,1,2\}, j\in \{3,4,5\}}$,} 
\STATE{$\wv$ has skew-symmetric rank 2 \AND
$$\wv=v_0\wedge v_1 \wedge v_2+v_3\wedge v_4 \wedge v_5.$$}
\ELSE\STATE{$\ker \mathcal{C}_{\wv}^{2,1}$ is as in \eqref{I2divisible},
\STATE $\wv$ has skew-symmetric rank 3 \AND
$\wv$ is as in (\ref{IV}):
$$\wv=v_0 \wedge v_1 \wedge v_2+v_0 \wedge v_3 \wedge v_4+v_1\wedge v_3 \wedge v_5$$ where $v_i=f_i$, $i=0, \ldots ,5$.}
\ENDIF
\end{algorithmic}
\end{breakablealgorithm}

{As already pointed out, the skew-symmetric rank classification of tensors in $\bigwedge^3\mathbb{C}^6$ is not new: it was already well known to C. Segre in 1917 (\cite{Segre}). The same was also done by  G.-C. Rota and J. Stein  in 1986 (\cite{rs}) with invariant theory perspective. We refer also to W. Chan who in 1998 wrote this classification in honor of Rota (\cite{Chan}). What we believe it's new in our approach is how to compute the skew-symmetric rank and a skew-symmetric minimal decomposition of any given element in $\bigwedge^3\mathbb{C}^6$.} 

\subsection{Trivectors in  $\mathbb{P}^6$}
If $n=6$ the classification of normal forms of trivectors is due to Schouten \cite{Sch}. In this case, in addition to the classes (\ref{II}) to (\ref{V}), there are five other classes of normal forms for the projective class of  a tensor $\wv\in \bigwedge^3\mathbb{C}^7$:
\begin{equation}\tag{VI}\label{VI} [ a\wedge q\wedge p +b\wedge r\wedge p+ c\wedge s\wedge p],
\end{equation}
\begin{equation}\tag{VII}\label{VII} [q\wedge r\wedge s+ a\wedge q\wedge p +b\wedge r\wedge p+ c\wedge s\wedge p],
\end{equation}
\begin{equation}\tag{VIII}\label{VIII} [a\wedge b\wedge c + q\wedge r\wedge s+ a\wedge q\wedge p ],
\end{equation}
\begin{equation}\tag{IX}\label{IX} [a\wedge b\wedge c + q\wedge r\wedge s+ a\wedge q\wedge p +b\wedge r\wedge p],
\end{equation}
\begin{equation}\tag{X}\label{X} [a\wedge b\wedge c + q\wedge r\wedge s+ a\wedge q\wedge p +b\wedge r\wedge p+ c\wedge s\wedge p].
\end{equation}
The containment diagram of the closures of the orbits of those normal forms is described in \cite{aop}. That diagram shows that (\ref{IX}) is a general element in $\sigma_3\bigl(\mathbb{G}(3,\mathbb{C}^7\bigr))$ hence (\ref{IX}) has skew-symmetric rank equal to 3. The fact that $\sigma_3\bigl(\mathbb{G}(3, \mathbb{C}^7)\bigr)$ is a defective hypersurface \cite{cgg, bdg,aop,b}  (meaning that by a simple count of parameters it is expected to fill the ambient space but it turns out to be a degree 7 hypersurface, see \cite{L,aop}) implies that there is an infinite number of ways to write a general element of  $\sigma_3\bigl(\mathbb{G}(3,\mathbb{C}^7)\bigr)$ as sum of 3 skew-symmetric rank 1 terms (\cite{cc1,cc2, bbc, bv}). It is in fact very easy to show that for any independent choice of $a,b,c,p,q,r,s \in \mathbb{C}^7$,  there always exist $\lambda_0  , \ldots , \lambda_5 \in \mathbb{C}$ and a basis $\{e_0, \ldots , e_6\}$ of $\mathbb{C}^7$ such that  (\ref{IX}) can be written for example as 
\begin{equation}\label{presentation:rank3}[e_0\wedge e_1 \wedge e_2 +e_3\wedge e_4 \wedge e_5  +e_6\wedge(e_0+\cdots +e_5)\wedge(\lambda_0e_0+ \cdots +\lambda_5e_5)]\end{equation}
which is actually a presentation of (\ref{IX}) as a skew-symmetric rank 3 tensor. 

\medskip

The closure of the orbit of \eqref{X} fills the ambient space so \eqref{X} corresponds to a general element of $\sigma_4\bigl(\mathbb{G}(3, \mathbb{C}^7)\bigr)$   hence such a tensor has skew-symmetric rank 4.  As above, if we want to see a presentation of (\ref{X}) as a skew-symmetric rank 4 tensor, it is sufficient to take an element of the form (\ref{presentation:rank3}) and add to it any random rank 1 element of $\bigwedge^3\mathbb{C}^7$. In this case, in order to get a decomposition of an element $\wv$ in the orbit of (\ref{X}), one can proceed as follows: a generic line through $\wv$ meets $\sigma_3\bigl(\mathbb{G}(3, \mathbb{C}^7)\bigr)$ in 7 points since $\sigma_3\bigl(\mathbb{G}(3, \mathbb{C}^7)\bigr)$ is a hypersurface of degree 7; pick a generic line among those joining $\mathbb{G}(3, \mathbb{C}^7)$ with $\wv$; if this line intersects  $\sigma_3\bigl(\mathbb{G}(3, \mathbb{C}^7)\bigr)$ in another point $\ww$ then the point on the Grassmannian together with $\ww$ will give a decomposition of $\wv$. The problem may be that, since the line joining the Grassmannian and $\wv$ is not generic, it may happen that the specific line chosen does not intersect  $\sigma_3\bigl(\mathbb{G}(3, \mathbb{C}^7)\bigr)$ in a different point, but then one has simply to try another line until one gets a distinct point on  $\sigma_3\bigl(\mathbb{G}(3, \mathbb{C}^7)\bigr)$.  Again for the reader familiar with numerical computations, the method used in \cite{bdhm} is working particularly well and fast for general tensors (see \cite[Section 4]{bdhm}).

\medskip 

The containment diagram of \cite{aop} shows that the orbit of (\ref{VIII}) strictly contains $\sigma_2\bigl(\mathbb{G}(3, \mathbb{C}^7)\bigr)$, therefore (\ref{VIII}) is not of skew-symmetric rank 2 and hence, since we have a presentation of skew-symmetric rank 3, it is actually of skew-symmetric rank 3. 

\medskip

For the orbit of \eqref{VI} we have a presentation with 3 summands, hence the rank is at most 3, but if the rank of \eqref{VI} was 2, the orbit of \eqref{VI} would be contained in the closure of the orbit of \eqref{V}, which is the open part of $\sigma_2\bigl(\mathbb{G}(\mathbb{P}^2, \mathbb{P}^6)\bigr)$, but, as \cite{aop} shows, this is not the case. Hence (\ref{VI}) has skew-symmetric rank 3. 

\medskip

We are left with (\ref{VII}). Its orbit is not contained in  $\sigma_2\bigl(\mathbb{G}(3, \mathbb{C}^7)\bigr)$ hence its rank is bigger than 2. Since we have a presentation of (\ref{VII}) with 4 summands, the containment diagram of the orbit closures is not giving any further information.
We need a more refined tool.

\begin{lemma}\label{increasing:degree:P6} 
Let $\wv=\wv_1+\wv_2+\wv_3\in \bigwedge^3 \mathbb{C}^7$ with $\wv_i\in \mathbb{G}( 3, \mathbb{C}^7)$ and let $[\vec \wv_i]$ be the planes in $\mathbb{P}^6$ corresponding to  $\wv_i$, $i=1,2,3$. Assume that $\langle [\vec \wv_1], [\vec \wv_2], [\vec \wv_3]\rangle=\mathbb{P}^6$. If there is an element $l\in \mathbb{C}^7$ such that $l\wedge \wv\in \bigwedge^4\mathbb{C}^7$ has skew-symmetric rank $r<2$, then there exists a basis $\{e_0, \ldots , e_6\}$ of $\mathbb{C}^7$ such that one of the following occurs:
\begin{enumerate}[(i)]
\item\label{T1} $\wv=e_0\wedge e_1\wedge e_2+e_3\wedge e_4\wedge e_5+e_0\wedge e_6\wedge (e_1+ \cdots +e_5)$, $l=e_0$ and $r=1$;
\item\label{T2} $\wv=e_0\wedge e_1\wedge e_2+e_0\wedge e_3\wedge e_4+e_0\wedge e_5\wedge e_6$, $l=e_0$ and $l\wedge \wv=0$;
\item\label{T3} $\wv=e_0\wedge e_1\wedge e_2+e_0\wedge e_3\wedge e_4+e_3\wedge e_5\wedge e_6$, either $l=e_0$ or $l=e_3$ and $r=1$.
\end{enumerate}   
\end{lemma}
 
\begin{proof} Clearly in all the three listed cases the skew-symmetric rank $r$ of $l\wedge \wv$ is smaller than 2. We only need to prove that these are the only possibilities.
Let $l=a_0e_0 + \cdots + a_6e_6$, $a_i \in \mathbb{C}$, $i=0, \ldots, 6$.

The only possibilities for the $[\vec \wv_i]$'s to span a $\mathbb{P}^6$ are:
\begin{itemize}
\item the $[\vec \wv_i]$'s don't intersect pairwise, in which case $\wv$ can be written as in \eqref{presentation:rank3};
\item exactly two of them meet at a point, in which case $\wv$ is as in \eqref{T1}; 
\item the three of them meet at the same point, in which case $\wv$ is as in \eqref{T2}; 
\item $[\vec \wv_1]\cap [\vec \wv_2]=P\neq Q=[\vec \wv_2]\cap [\vec \wv_3]$, in which case $\wv$ is as in \eqref{T3}.
\end{itemize}

If  $r<2$ then $\dim (\ker \mathcal{C}_{l\wedge \wv}^{1,3}) \geq \dim (\mathbb{C}^7)- \dim (\vec \wv')=3$ where $\wv'$ is any skew-symmetric rank 1 element in $\bigwedge^4\mathbb{C}^7$.
Therefore we need to impose that all the ${5\times 5}$ minors of $\mathcal{C}_{l\wedge \wv}^{1,3}$ vanish. A straightforward computation shows that the only solution for tensors of the form \eqref{presentation:rank3} is $l=0$, while for the other cases  $l$ is as in the statement.
\end{proof}

Remark that if ${\wv\in \bigwedge^3\mathbb{C}^7}$ is as in case \eqref{VII}, then the skew-symmetric rank of $p\wedge \wv$ is $1$ and $\wv$ is either as in \eqref{T1} or as in \eqref{T3}. 

\begin{lemma}\label{rank3} Let $\wv=\wv_1+\wv_2+\wv_3\in \bigwedge^3 \mathbb{C}^7$ with $\wv_i\in \mathbb{G}( 3, \mathbb{C}^7)$ be a minimal presentation of skew-symmetric rank 3 of $\wv$. Let $[\vec \wv_i]$ the planes in $\mathbb{P}^6$ corresponding to  $\wv_i$, $i=1,2,3$ and assume that $\langle [\vec \wv_1], [\vec \wv_2], [\vec \wv_3]\rangle=\mathbb{P}^6$. Then $\wv$ lies  in one of the orbits \eqref{VI}, \eqref{VII}, \eqref{VIII} and \eqref{IX}.
\end{lemma}

\begin{proof} Since the span of $[\vec \wv_1], [\vec \wv_2] , [\vec \wv_3]$ is fixed to be a $\mathbb{P}^6$ we may consider another invariant that is preserved by the action of $SL(7)$, namely the intersection of the $[\vec \wv_i]$'s. 
We list all possible configurations for the $ [\vec \wv_i]$'s with respect to their intersections, and write their general element in each case.  First of all observe that $[\vec \wv_i]\cap [\vec \wv_j]$ is at most a point since otherwise $\wv_1+\wv_2+\wv_3$ would not be a minimal presentation of skew-symmetric rank 3 for $\wv$; in fact $\wv_i+\wv_j$ would have skew-symmetric rank 1.
We assume that $P_1,P_2,P_3\in \mathbb{P}^6$ are distinct points, $\{e_0, \ldots , e_6\}$ is a basis of $\mathbb{C}^7$ and $l\in \mathbb{C}^8$ is general enough.
\begin{enumerate}
\item\label{s3} None of them intersect each other
\[ e_0\wedge e_1\wedge e_2 + e_3\wedge e_4\wedge e_5 + e_6\wedge (e_0+ \cdots +e_6) \wedge l;\]
\item\label{differentVII} Exactly two of them intersect in one point
\[ e_0\wedge e_1\wedge e_2 + e_0\wedge e_3\wedge e_4 + e_5\wedge e_6\wedge (e_0+ \cdots +e_6)
;\]
\item\label{almostVI}  $ [\vec \wv_1]\cap  [\vec \wv_2]\cap [\vec \wv_3]=P_1$ 
\[ e_0\wedge e_1\wedge e_2 + e_0\wedge e_3\wedge e_4 + e_0\wedge e_5\wedge e_6
;\]
\item\label{almostVIII} $ [\vec \wv_1]\cap  [\vec \wv_2]=P_1$ and $[\vec \wv_2] \cap [\vec \wv_3]=P_2$ 
\[ e_0\wedge e_1\wedge e_2 + e_0\wedge e_3\wedge e_4 + e_3\wedge e_5\wedge e_6
.\]
\end{enumerate}
Observe that if $[\vec \wv_1]\cap [\vec \wv_2]=P_1$, $[\vec \wv_2]\cap [\vec \wv_3]=P_2$ and $[\vec \wv_1]\cap [\vec \wv_3]=P_3$ then $\langle [\vec \wv_1], [\vec \wv_2] , [\vec \wv_3]\rangle$ will be at most a $\mathbb{P}^5$ contradicting our hypothesis.

We can easily recognize that \eqref{s3} is the same presentation as \eqref{presentation:rank3} so it describes the same orbit of
\eqref{IX}. Moreover \eqref{almostVI} is the same presentation as \eqref{VI}, and \eqref{almostVIII} is the same presentation of \eqref{VIII}. 

Each of these elements must belong to a different orbit in \cite{gur1}. We have just seen that the elements in the orbits listed in the statement have skew-symmetric rank at least $3$. Since we have exactly $4$ intersection configurations and $4$ orbits there must be a one to one correspondence.  Therefore \eqref{differentVII} must be an element in the same orbit of \eqref{VII}. 
\end{proof}

It's very easy to check that if rk$(\wv)=1$ then $\dim \ker \mathcal{C}_{\wv}^{1,2}=4$, while if rk$(\wv)=2$ then $\dim \ker \mathcal{C}_{\wv}^{1,2}=1,2$ while in all other cases $\dim \ker \mathcal{C}_{\wv}^{1,2}=0, \dim \ker \mathcal{C}_{\wv}^{2,1}=14$.
Now we need to distinguish the cases of skew-symmetric rank 3 form the one of skew-symmetric rank 4. Since all the orbit closures of those of skew-symmetric rank 3 are contained in $\sigma_3(\mathbb{G}(3, \mathbb{C}^7))$ whose equations is classically know (see \cite{L} and \cite[Thm. 5.1]{aop}) one just needs to test the given tensor on that equation.

We are now ready to write the algorithm for the skew-symmetric rank decomposition in the case of $\bigwedge^3\mathbb{C}^7$.
\medskip

\begin{breakablealgorithm}
\renewcommand{\thealgorithm}{2}
\caption{: Algorithm for the skew-symmetric rank and a decomposition of  an element  in $\bigwedge^3\mathbb{C}^7$. }
\leftline{INPUT: $\wv\in \bigwedge^3 \mathbb{C}^7$.}
\leftline{OUTPUT: Decomposition and skew-symmetric rank of $\wv$}
    \begin{algorithmic}[1]
\STATE Compute $\dim \ker \mathcal{C}_{\wv}^{1,2}$;
\IF{$\dim \ker \mathcal{C}_{\wv}^{1,2}=4$}
\STATE{$\wv$ has skew-symmetric rank 1 \AND
$$\wv=v_0\wedge v_1 \wedge v_2,$$ 
where $\langle v_0, v_1, v_2 \rangle = (\ker \mathcal{C}_{\wv}^{1,2})^{\perp}$.}
\ENDIF
\IF{$\dim \ker \mathcal{C}_{\wv}^{1,2}=2$}
\STATE{$\wv$ has skew-symmetric rank 2, \AND
$$\wv=v_0\wedge v_1 \wedge v_2+v_0\wedge v_3\wedge v_4,$$ 
 		where $\langle v_0, v_1, v_2 ,v_3, v_4\rangle = (\ker \mathcal{C}_{\wv}^{1,2})^{\perp}:=K$, $v_0\in \ker (K \stackrel{\wedge 		\wv}	{\longrightarrow} \bigwedge^4 \mathbb{C}^7)$, \AND $\langle v_1^*\wedge v_2^*-v_3^*\wedge v_4^*, (v_i^*\wedge 			v_j^*)_{i\in \{1,2\}, j\in \{3,4\}}\rangle\subset\ker \mathcal{C}_{\wv}^{2,1}$.}		
\ENDIF
\IF{$\dim \ker \mathcal{C}_{\wv}^{1,2}=1$}
\STATE{$\wv$ has skew-symmetric rank 2 \AND
$$\wv=v_0\wedge v_1 \wedge v_2+v_3\wedge v_4\wedge v_5,$$ where $\langle v_0, v_1, v_2 ,v_3, v_4, v_5\rangle = (\ker 		\mathcal{C}_{\wv}^{1,2})^{\perp}$ \AND $\langle v_i^*\wedge v_j^*\rangle_{i\in \{0,1,2\}, j\in \{3,4,5\}}\subset\ker 
\mathcal{C}_{\wv}^{2,1}$.}
\ENDIF
\IF{If $\dim \ker \mathcal{C}_{\wv}^{1,2}=0$}
\STATE\label{checksecant}{check wether $\wv\in\sigma_3(\mathbb{G}(3, \mathbb{C}^7))$ using \cite[Thm. 5.1]{aop}},
\IF\TRUE\label{checkedsecantyes}
\STATE go to Step \ref{AOP}
\ELSE\STATE\label{checkedsecantno} go to Step \ref{RANK4}
\ENDIF
\ENDIF
\STATE\label{AOP}{\footnotesize{\COMMENT{By Steps \ref{checksecant} and \ref{checkedsecantyes} we know that are in the cases in which $\wv\in\sigma_3(\mathbb{G}(3, \mathbb{C}^7))$ hence} }}
\STATE {\footnotesize{\COMMENT{In order to have the presentation of $\wv$ with 3 summands}}} Compute the Kernel and  the Image of the following multiplication map by $\wv$:
\begin{equation}\label{multmap}\mathbb{C}^7 \stackrel{\wedge \wv}{\longrightarrow} {\textstyle \bigwedge^4 \mathbb{C}^7},
\end{equation}
\IF{the kernel of \eqref{multmap} is non-zero}
\STATE $\wv$ has skew-symmetric rank 3 \AND \[
\wv=v_0\wedge v_1\wedge v_2 + v_0\wedge v_3\wedge v_4 + v_0\wedge v_5\wedge v_6
        		\]
as in Lemma \ref{rank3} item \eqref{almostVI}, where $v_0$ is a generator for the kernel of \eqref{multmap}, \AND $\langle 			v_1^*\wedge v_2^*-v_3^*\wedge v_4^*, v_1^*\wedge v_2^*-v_5^*\wedge v_6^*, (v_i^*\wedge v_j^*)_{i\in \{1,2\}, j\in 				\{3,4,5,6\}}, (v_i^*\wedge v_j^*)_{i\in \{3,4\}, j\in \{5,6\}}\rangle\subset\ker \mathcal{C}_{\wv}^{2,1}$;
\ENDIF
\IF{the kernel of \eqref{multmap} is zero \AND the image of \eqref{multmap} meets the Grassmannian in two points} 
\STATE $\wv$ has skew-symmetric rank 3 \AND
\[
\wv= v_0\wedge v_1\wedge v_2 + v_0\wedge v_3\wedge v_4 + v_3\wedge v_5\wedge v_6
\]
as in  Lemma \ref{rank3} item  \eqref{almostVIII}, where $v_0$ and $v_3$ are pre-images of the two points in the Grassmannian, 	\AND $\langle v_1^*\wedge v_2^*-v_3^*\wedge v_4^*, v_0^*\wedge v_4^*+v_5^*\wedge v_6^*, (v_0^*\wedge v_i^*)_{i\in \{5,6\}}, 	(v_i^*\wedge v_j^*)_{i\in \{1,2\}, j\in \{3,4,5,6\}}, (v_4^*\wedge v_i^*)_{i\in \{5,6\}}\rangle\subset\ker \mathcal{C}_{\wv}^{2,1}$;
\ENDIF
\IF{the kernel of \eqref{multmap} is zero \AND the image of $\wedge \wv$ meets the Grassmannian in one point}
\STATE $\wv$ has skew-symmetric rank 3 \AND
\[
\wv= v_0\wedge v_1\wedge v_2 + v_0\wedge v_3\wedge v_4 + v_5\wedge v_6\wedge (v_0+ \cdots +v_6)
\]
as in  Lemma \ref{rank3} item  \eqref{differentVII}, where $v_0$ is a pre-image of the point in the Grassmannian, and $\langle 		v_1^*\wedge v_2^*-v_3^*\wedge v_4^*,  ((v_0^*-v_i^*)\wedge v_j^*)_{i\in \{1,2,3,4\}, j\in \{5,6\}}, (v_i^*\wedge v_j^*)_{i\in \{1,2\}, j\in 	\{3,4\}} \rangle\subset\ker \mathcal{C}_{\wv}^{2,1}$;
\ENDIF
\IF{the kernel of \eqref{multmap} is zero \AND the image of the map $\wedge \wv$ does not meet the Grassmannian}
\STATE\label{step_sigma_3} $\wv$ has skew-symmetric rank 3 \AND \[
\wv= v_0\wedge v_1\wedge v_2 + v_3\wedge v_4\wedge v_5 + v_6\wedge (v_0+ \cdots +v_6) \wedge l
\]
as in  Lemma \ref{rank3} item  \eqref{s3}, where ${v_0,\ldots,v_6,l}$ are such that the preimage of the Grassmannian by the map
	\begin{equation}\label{multmap25}
	{\textstyle \bigwedge^2\mathbb{C}^7} \stackrel{\wedge \wv}{\longrightarrow} 
	  {\textstyle \bigwedge^5 \mathbb{C}^7}
	\end{equation}
	is the set of elements that can be written
        either as
$$(a_0v_0 + a_1v_1 + a_2v_2) \wedge (a_3v_3 + a_4v_4 + a_5v_5)$$
or as
$$(a_0v_0 + a_1v_1 + a_2v_2) \wedge (a_6v_6 + a_7(v_0+\cdots+v_6) + a_8l)$$
or as
$$(a_3v_3 + a_4v_4 + a_5v_5) \wedge (a_6v_6 + a_7(v_0+\cdots+v_6) + a_8l)$$
\ENDIF
\STATE\label{RANK4} {\footnotesize{\COMMENT{By Steps \ref{checksecant} and \ref{checkedsecantno} we know that we are in the case in which $\wv\not\in\sigma_3(\mathbb{G}(3, \mathbb{C}^7))$ hence} }}

$\wv$ has skew-symmetric rank 4 \AND 
$\wv$ can be written as follows:
\[
v_0\wedge v_1 \wedge v_2 +v_3\wedge v_4 \wedge v_5  +v_6\wedge (v_0+\cdots +v_6)\wedge l_0+  l_1\wedge l_2 \wedge l_3
\]
where 
${v_0,\ldots,v_6,l_0\ldots,l_3}$ are such that $l_1\wedge l_2 \wedge l_3$ is a generic element in the Grassmannian, and $v_0,\ldots,v_6,l_0$ are obtained running Step \ref{step_sigma_3} of the present algorithm on one of the 7 points of $\langle \wv, l_1\wedge l_2 \wedge l_3\rangle \cap  \sigma_3(\mathbb{G}(3, \mathbb{C}^7))$. If
 something goes wrong change either the starting point $l_1\wedge l_2 \wedge l_3$ or one of the seven points on $\langle \wv, l_1\wedge l_2 \wedge l_3\rangle \cap  \sigma_3(\mathbb{G}(3, \mathbb{C}^7))$ until the algorithm ends.
    \end{algorithmic}
\end{breakablealgorithm}

\subsection{Trivectors in  $\mathbb{P}^7$} 
If $n=7$ there are 22 orbits in the projective space. The first nine normal forms for a trivector in $\bigwedge^3 \mathbb{C}^8$, with some meaning in the projective space, are the same as above from (\ref{II}) to (\ref{X}), the other are described in \cite[Chap. VII, \S 35.4]{gur1}:

\begin{equation}\tag{XI}\label{XI}
[	 a\wedge q\wedge p+ 
b\wedge r\wedge p +
 c\wedge s\wedge p+ 
 c\wedge r\wedge t]
\end{equation}	

\begin{equation}\tag{XII}\label{XII}
[  q\wedge r\wedge s+
 a\wedge q\wedge p+ 
 b\wedge r\wedge p +
 c\wedge s\wedge p+ 
 c\wedge r\wedge t]
\end{equation}	

\begin{equation}\tag{XIII}\label{XIII}
[  a\wedge b\wedge c+
 q\wedge r\wedge s+
 a\wedge q\wedge p+ 
 c\wedge r\wedge t]
\end{equation}	

\begin{equation}\tag{XIV}\label{XIV}
[  a\wedge b\wedge c+
 q\wedge r\wedge s+
 a\wedge q\wedge p+ 
 b\wedge r\wedge p +
 c\wedge r\wedge t]
\end{equation}	

\begin{equation}\tag{XV}\label{XV}
[  a\wedge b\wedge c+
 q\wedge r\wedge s+
 a\wedge q\wedge p+ 
 b\wedge r\wedge p +
 c\wedge s\wedge p+ 
 c\wedge r\wedge t]
\end{equation}	

\begin{equation}\tag{XVI}\label{XVI}
[ a\wedge q\wedge p+ 
 b\wedge s\wedge t+
 c\wedge r\wedge t]
\end{equation}	

\begin{equation}\tag{XVII}\label{XVII}
[ a\wedge q\wedge p+ 
 b\wedge r\wedge p +
 b\wedge s\wedge t+
 c\wedge r\wedge t]
\end{equation}	

\begin{equation}\tag{XVIII}\label{XVIII}
[ q\wedge r\wedge s+
 a\wedge q\wedge p+ 
 b\wedge r\wedge p +
 b\wedge s\wedge t+
 c\wedge r\wedge t]
\end{equation}	

\begin{equation}\tag{XIX}\label{XIX}
[ a\wedge q\wedge p+ 
 b\wedge r\wedge p +
 c\wedge s\wedge p+ 
 b\wedge s\wedge t+
 c\wedge r\wedge t]
\end{equation}	

\begin{equation}\tag{XX}\label{XX}
[ q\wedge r\wedge s+
 a\wedge q\wedge p+ 
 b\wedge r\wedge p +
 c\wedge s\wedge p+ 
 b\wedge s\wedge t+
 c\wedge r\wedge t]
\end{equation}	

\begin{equation}\tag{XXI}\label{XXI}
[ a\wedge b\wedge c+
 q\wedge r\wedge s+
 a\wedge q\wedge p+ 
 b\wedge s\wedge t
 ]
\end{equation}	

\begin{equation}\tag{XXII}\label{XXII}
[ a\wedge b\wedge c+
 q\wedge r\wedge s+
 a\wedge q\wedge p+ 
 b\wedge r\wedge p +
 b\wedge s\wedge t+
 c\wedge r\wedge t]
\end{equation}	

\begin{equation}\tag{{XXIII}}\label{XXIII}
[ a\wedge b\wedge c+
 q\wedge r\wedge s+
 a\wedge q\wedge p+ 
 b\wedge r\wedge p +
 c\wedge s\wedge p+ 
 b\wedge s\wedge t+
 c\wedge r\wedge t]
\end{equation}	

The containment diagram of the orbit closure is described in {\cite{Dj,bv}}. 

\medskip

From that diagram and for the normal forms that we have, it's easy to see that \eqref{XVI}  has skew-symmetric rank 3 since it's orbit closure is not contained in $\sigma_2(\mathbb{G}(3, \mathbb{C}^8))$  and we have a presentation with $3$ summands. 

\medskip 

Similarly an element (\ref{XXI}) has skew-symmetric rank  4 since  it's orbit closure is not contained in $\sigma_3(\mathbb{G}(3, \mathbb{C}^8))$  and we have a presentation with $4$ summands.

\medskip

Moreover for an element in   (\ref{XIX}) the skew-symmetric rank is 3 because the closure of its orbit is $\sigma_3(\mathbb{G}(3, \mathbb{C}^8))$. 

\medskip

Finally for an element in  (\ref{XXIII}) the skew-symmetric  rank is 4 because the closure of its orbit is $\sigma_4(\mathbb{G}(3, \mathbb{C}^8))$ which fills the ambient space.

\medskip

The normal form of \eqref{XX} is a presentation of 6 summands, 5 of which coincides with the presentation of the normal form of \eqref{XIX} which has skew-symmetric rank 3. Therefore since the orbit closure of \eqref{XX} is not contained in $\sigma_3(\mathbb{G}(3, \mathbb{C}^8))$ the skew-symmetric rank of it is exactly 4.

\medskip

{The normal form of \eqref{XII} coincides with the normal form of \eqref{VII} plus $c\wedge r \wedge t$, therefore its skew-symmetric rank is at most 3+1, but since  the orbit closure of  \eqref{XII}  is not contained in $\sigma_3(\mathbb{G}(3, \mathbb{C}^8))$ the skew-symmetric rank of it is exactly 4.}

\medskip

{The normal form of \eqref{XIV} coincides with the normal form of \eqref{IX} plus $c\wedge r \wedge t$, therefore its skew-symmetric rank is at most 3+1, but since  the orbit closure of  \eqref{XII}  is not contained in $\sigma_3(\mathbb{G}(3, \mathbb{C}^8))$ the skew-symmetric rank of it is exactly 4.}

\medskip

Analogously the normal form of \eqref{XV} has a presentation with only one more element than the normal form of \eqref{X} so the skew-symmetric rank of the elements in \eqref{XV} is either 4 or 5 even though we have a presentation with 6 summands.

\medskip

{By reordering the variables, it is not difficult to see that \eqref{XVIII} is nothing else then \eqref{VII} plus $a\wedge q \wedge p$, therefore its skew-symmetric rank is at most 3+1, but since  the orbit closure of  \eqref{XVIII}  is not contained in $\sigma_3(\mathbb{G}(3, \mathbb{C}^8))$ the skew-symmetric rank of it is exactly 4.}
\medskip

{The normal form of \eqref{XXII} can be written by subtracting $c\wedge s \wedge p$ to the normal form of \eqref{XXIII} which has rank 4
 therefore the skew-symmetric rank of the elements in the orbit of \eqref{XXII} is either 4 or 5.}
\medskip

\begin{lemma}\label{increasing:degree:P7} 
Let ${\wv=\wv_1+\wv_2+\wv_3\in \bigwedge^3 \mathbb{C}^8}$ with ${\wv_i\in \mathbb{G}( 3, \mathbb{C}^8)}$ be a minimal presentation of skew-symmetric rank 3 of $\wv$. Let ${[\vec \wv_i]}$ be the planes in $\mathbb{P}^7$ corresponding to  $\wv_i$, ${i=1,2,3}$, and assume that ${\langle [\vec \wv_1], [\vec \wv_2], [\vec \wv_3]\rangle=\mathbb{P}^7}$. Then $\wv$ lies either in the orbit of \eqref{XIX} or in the orbit of \eqref{XVI}.
\end{lemma}

\begin{proof}
There always  exists a basis $\{e_0, \ldots , e_7\}$ of $\mathbb{C}^8$ such that $\wv$ can be written as follows:
\begin{itemize}
\item $\wv=e_0\wedge e_1\wedge e_2+e_3\wedge e_4\wedge e_5+e_6\wedge e_7\wedge (e_1+ \cdots +e_5)$,
\item $\wv=e_0\wedge e_1\wedge e_2+e_3\wedge e_4\wedge e_5+e_0\wedge e_6\wedge e_7$. 
\end{itemize}
The first case corresponds to a generic element of $\sigma_3(\mathbb{G}(3, \mathbb{C}^8))$ which is the orbit closure of \eqref{XIX}; the second one is the normal form of the orbit of \eqref{XVI}.
\end{proof}


By Lemma \ref{increasing:degree:P7}, the elements 
of orbits  \eqref{XI}, \eqref{XIII} and \eqref{XVII} have skew-symmetric rank 4.

\begin{remark}\label{cases45}  We are left with 
\eqref{XV}, 
and \eqref{XXII} where we still have to determine if the skew-symmetric rank is 4 or 5. In fact for all the other cases from \eqref{XI} to \eqref{XXIII} we have already shown that the skew-symmetric rank is 4 except for \eqref{XVI} and \eqref{XIX} where the skew-symmetric rank is 3.
\end{remark}

For the case of \eqref{XV} we follow the idea of \cite{W}.

\begin{remark}\label{West} 
Let $\wv\in \bigwedge^3\mathbb{C}^8$ be a tensor of skew-symmetric rank 4 in 8 essential variables, $a, b, c, p, q, r, s, t$. Then at least one of the terms in any skew-symmetric rank  4 representation of $\wv$ must contain a factor of the form $(t-v)$ for some ${v\in  \langle a, b, c, p, q, r, s \rangle}$  and the skew-symmetric rank of $(t-v) \wedge \wv$ is 3. To see this, write a skew-symmetric rank  4 representation of $\wv$: 
\[\wv = \wv_1 + \wv_2 + \wv_3 + \wv_4.\]
Observe that $t$ must occur in at least one summand, say ${\wv_1=(\alpha t + w_1)\wedge w_2 \wedge w_3}$, with ${\alpha\ne0}$, ${w_1\in \langle a, b, c, p, q, r, s \rangle}$; we may then rewrite ${\wv_1=(t + \alpha^{-1} w_1)\wedge \alpha w_2 \wedge w_3}$, and make ${-v=\alpha^{-1} w_1}$. 
\end{remark}

\begin{proposition}[Westwick] 
The skew-symmetric rank of an element in the orbit of  \eqref{XV} is 5.
\end{proposition}
\begin{proof} 
Let ${v\in \langle a, b, c, p, q, r, s \rangle}$, consider the vector ${(t-v)\in \langle a, b, c, p, q, r, s,t\rangle}$ and write ${\eqref{XV}=\eqref{X}+c\wedge r \wedge t}$. Now ${(t-v)\wedge \eqref{XV}=t\wedge\bigl(\eqref{X}-v\wedge c\wedge r\bigr) -v\wedge\eqref{X}}$. It is therefore sufficient to show that the skew-symmetric rank of ${\eqref{X}+v\wedge c\wedge r}$ is at least 4 for any ${v\in \langle a, b, c, p, q, r, s \rangle}$ and we conclude by Remark \ref{West}.

The tensor ${\eqref{X}+v\wedge c\wedge r}$ has 7 essential variables, moreover we can pick a vector ${v=\alpha_1a+\alpha_2 b+ \alpha_3 p +\alpha_4 q + \alpha_5 s \in \langle a, b, p, q, s \rangle}$ and then it is a straightforward computation to check that there is no choice of the $\alpha_i$'s that annihilates the equation of $\sigma_3\bigl(\mathbb{G}(3, \mathbb{C}^7)\bigr)$.\end{proof}

Finally, we are left with case \eqref{XXII} where we have to understand if the skew\-\mbox{-symmetric} rank is either 4 or 5. Westwick in \cite{W} exhibited a decomposition with $4$ summands, and this clearly suffices to say that the skew-symmetric rank of an element in the orbit of \eqref{XXII} is 4. Anyway the brute-force computations, with some clever observation, show that there is a family of projective  dimension a least $10$ of skew-symmetric rank $4$ decompositions for the elements in the orbit of \eqref{XXII}.

The trivial brute-force computation requires to find a solution of
\begin{multline*}
 a\wedge b\wedge c+
 q\wedge r\wedge s+
 a\wedge q\wedge p+ 
 b\wedge r\wedge p +
 b\wedge s\wedge t+
 c\wedge r\wedge t= \\
=\sum_{i=1}^4 \lambda_i v_i\wedge w_i \wedge u_i
\end{multline*}
with
\begin{align*}
v_i &= l_{i,1}a+l_{i,2}b+l_{i,3}c+ l_{i,4}p+ l_{i,5}q +l_{i,6}r +l_{i,7}s+l_{i,8}t, \\
w_i &= m_{i,1}a+m_{i,2}b+m_{i,3}c+ m_{i,4}p+ m_{i,5}q +m_{i,6}r +m_{i,7}s+m_{i,8}t, \\
u_i &= n_{i,1}a+n_{i,2}b+n_{i,3}c+ n_{i,4}p+ n_{i,5}q +n_{i,6}r +n_{i,7}s+n_{i,8}t,
\end{align*}
and $l_{i,j},m_{i,j},n_{i,j}\in \mathbb{C}$ for ${i=1, \dots, 4}$ and ${j=1,\ldots,8}$, which is almost impossible to solve in a reasonable amount of time.

Then one may look at the structure of \eqref{XXII}: it is the sum of $6$ summands, each one of them represents a $\mathbb{P}^2 \subset \mathbb{P}^7$. We draw each of this $\mathbb{P}^2$'s as a triangle where the vertices are the generators: e.g.\ Figure \ref{figure:abc} represents the projectivization of $\langle a,b,c \rangle$.

\begin{figure}[!htp]
\includegraphics[width=.3\textwidth]{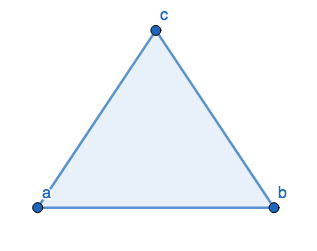}
{ \caption{\small{The projectivization of $\langle a,b,c \rangle$.}}\label{figure:abc}} 
\end{figure}
If we draw all the six $\mathbb{P}^2$'s appearing in the decomposition of \eqref{XXII} according with this technique, the graph that we get is represented in Figure \ref{figure:6}.
\begin{figure}[!htp]
\includegraphics[width=.7\textwidth]{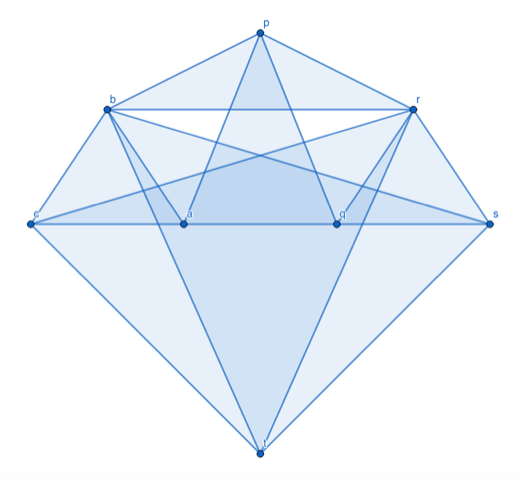}
{ \caption{\small{Picture representing the planes in \eqref{XXII}.}}\label{figure:6}} 
\end{figure}
Now, consider for example the vertex $b$; for each of the $\mathbb{P}^2$'s appearing in \eqref{XXII} where $b$ does not appear as a generator (i.e. $\langle q,r,s \rangle$, $\langle a,q,p \rangle$ and $\langle c,r,t \rangle$) there are at least two edges linking $b$ with that $\mathbb{P}^2$: the edges $bp$ and $bs$ link $b$ with $\langle q,r,s\rangle$, etc. The same phenomenon occurs for $r$, but does not for the other vertices. Therefore we draw new edges in order to link any vertex with any of the $\mathbb{P}^2$ where it does not appear as a generator with at least $2$ edges. We add:
\begin{itemize}
\item \grey{$ar$} and \grey{$at$} in order to connect $a$ twice with the triangles of \eqref{XXII} not involving $a$;
\item \red{$cp$} and \red{$cs$} in order to connect $c$ twice with the triangles of \eqref{XXII} not involving $c$;
\item \cornellred{$ps$} in order to connect $p$ twice with the triangle  of \eqref{XXII} not involving $p$;
\item \green{$bq$} and \green{$qt$} in order to cnnect $q$ twice with the triangles of \eqref{XXII} not involving $q$.
\end{itemize}
These new edges suffice to make all vertices connected at least twice with all the triangles of \eqref{XXII} not involving the vertex considered. We have drawn the new graph in Figure \ref{figure:complete:graph}; it's remarkable that the symmetry of the graph is preserved and that now every vertex has 6 edges starting from it.
\begin{figure}[!htp]
\includegraphics[width=.6\textwidth]{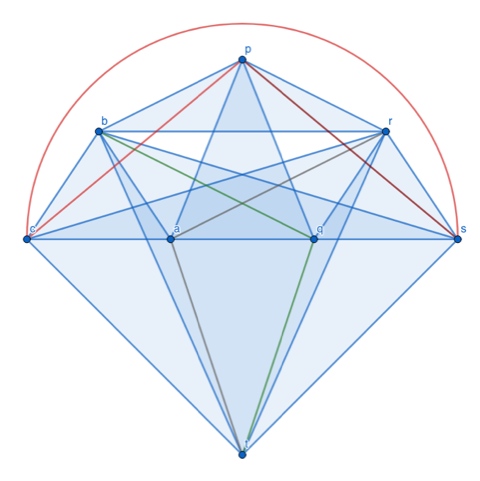}
{ \caption{\small{Picture representing the planes in \eqref{XXII}, with the extra edges.}}\label{figure:complete:graph}} 
\end{figure}
Now, taking into account the edges that we have added, we look for a decomposition of \eqref{XXII} of the following form:
\begin{equation}\label{simplified:system}\sum_{i_1}^4 \hbox{\grey{$(l_{i,1}a+l_{i,2}r+l_{i,3}t)$}}\wedge \hbox{\green{$(m_{i,1}b+m_{i,2}q+m_{i,3}t)$}} \wedge \hbox{\red{$(n_{i,1}c+n_{i,2}s+n_{i,3}p)$}}\end{equation}
with $l_{i,j},m_{i,j}\in \mathbb{C}$ for $i=1,\ldots, 4$ and ${j=1, 2,3}$. We computed the ideal of the solution of the system  \eqref{simplified:system} both with Macaulay2 (\cite{Macaulay2}) and with Bertini (\cite{Bertini}) and we got projective dimension $10$ and degree  $2556$. The ideal we got with \cite{Macaulay2} has the following generators:\\
\small{$l_{(4,1)}m_{(4,3)}n_{(4,1)},\\l_{(4,3)}m_{(4,2)}n_{(4,1)},\\l_{(4,1)}m_{(4,2)}n_{(4,1)},\\l_{(3,3)}m_{(3,1)}n_{(3,1)}-l_{(4,3)}m_{(4,1)}n_{(4,1)},\\l_{(3,1)}m_{(3,1)}n_{(3,1)}-l_{(4,1)}m_{(4,1)}n_{(4,1)}-1,\\l_{(2,2)}m_{(2,3)}n_{(2,1)}-l_{(4,2)}m_{(4,3)}n_{(4,1)}-1,\\l_{(2,2)}m_{(2,2)}n_{(2,1)}-l_{(4,2)}m_{(4,2)}n_{(4,1)},\\l_{(2,2)}m_{(2,1)}n_{(2,1)}+l_{(3,2)}m_{(3,1)}n_{(3,1)}-l_{(4,2)}m_{(4,1)}n_{(4,1)},\\l_{(1,2)}m_{(1,3)}n_{(1,3)}+l_{(2,2)}m_{(2,3)}n_{(2,3)}-l_{(4,2)}m_{(4,3)}n_{(4,3)},\\l_{(1,1)}m_{(1,3)}n_{(1,3)}-l_{(4,1)}m_{(4,3)}n_{(4,3)},\\l_{(1,3)}m_{(1,2)}n_{(1,3)}-l_{(4,3)}m_{(4,2)}n_{(4,3)},\\l_{(1,2)}m_{(1,2)}n_{(1,3)}+l_{(2,2)}m_{(2,2)}n_{(2,3)}-l_{(4,2)}m_{(4,2)}n_{(4,3)},\\l_{(1,1)}m_{(1,2)}n_{(1,3)}-l_{(4,1)}m_{(4,2)}n_{(4,3)}-1,\\l_{(1,3)}m_{(1,1)}n_{(1,3)}+l_{(3,3)}m_{(3,1)}n_{(3,3)}-l_{(4,3)}m_{(4,1)}n_{(4,3)},\\l_{(1,2)}m_{(1,1)}n_{(1,3)}+l_{(2,2)}m_{(2,1)}n_{(2,3)}+l_{(3,2)}m_{(3,1)}n_{(3,3)}-l_{(4,2)}m_{(4,1)}n_{(4,3)}+1,l_{(1,1)}m_{(1,1)}n_{(1,3)}+l_{(3,1)}m_{(3,1)}n_{(3,3)}-l_{(4,1)}m_{(4,1)}n_{(4,3)},\\l_{(1,2)}m_{(1,3)}n_{(1,2)}+l_{(2,2)}m_{(2,3)}n_{(2,2)}-l_{(4,2)}m_{(4,3)}n_{(4,2)},\\l_{(1,1)}m_{(1,3)}n_{(1,2)}-l_{(4,1)}m_{(4,3)}n_{(4,2)},\\l_{(1,3)}m_{(1,2)}n_{(1,2)}-l_{(4,3)}m_{(4,2)}n_{(4,2)},\\l_{(1,2)}m_{(1,2)}n_{(1,2)}+l_{(2,2)}m_{(2,2)}n_{(2,2)}-l_{(4,2)}m_{(4,2)}n_{(4,2)}+1,l_{(1,1)}m_{(1,2)}n_{(1,2)}-l_{(4,1)}m_{(4,2)}n_{(4,2)},\\l_{(1,3)}m_{(1,1)}n_{(1,2)}+l_{(3,3)}m_{(3,1)}n_{(3,2)}-l_{(4,3)}m_{(4,1)}n_{(4,2)}-1,\\l_{(1,2)}m_{(1,1)}n_{(1,2)}+l_{(2,2)}m_{(2,1)}n_{(2,2)}+l_{(3,2)}m_{(3,1)}n_{(3,2)}-l_{(4,2)}m_{(4,1)}n_{(4,2)},\\l_{(1,1)}m_{(1,1)}n_{(1,2)}+l_{(3,1)}m_{(3,1)}n_{(3,2)}-l_{(4,1)}m_{(4,1)}n_{(4,2)}$}.

\bigskip

We have then proved the following.

\begin{theorem} Let ${t\in \bigwedge^3 \mathbb{C}^8}$  be a skew-symmetric tensors with $8$ essential variables. Then the skew-symmetric rank of $t$ is $4$ except in the following cases:
\begin{itemize}
\item $t$ is of type  either \eqref{XVI} or \eqref{XIX}, in which cases its skew-symmetric rank is $3$,
\item $t$ is of type  \eqref{XV}, in which case its skew-symmetric rank is $5$.
\end{itemize}
\end{theorem}

%
%

\bigskip

We list all cases of trivectors in the following table, where we use the the Roman numbering and write the normal forms (NF) by Gurevich \cite{gur1}. We add the skew-symmetric rank decomposition (SD) and the skew-symmetric rank (SSR). In cases XII, XIV, XV, XVIII, XX, and XXII, the SD is presented using the same basis as the NF, as in \cite{W}. In all other cases, we use linearly independent vectors ${v_0,\ldots, v_7}$, and general linear forms ${m_0, \ldots, m_3\in \langle v_0,\ldots, v_6 \rangle}$ and ${l_0, l_1, l_2\in \langle v_0,\ldots, v_7 \rangle}$ 

Note that for IX and X the presentation of \cite{W} allows the two following SD:  $ab(c-p) + (a-r)(b+q)p + rq(p-s)$, for \eqref{IX}, and  $aqp + (b+s)(r-c)p + (a+p)bc + (p+q)rs$, for  \eqref{X}.

w

\newcommand{\nf}{ & NF: }
\newcommand{\sd}{\\ & SD: }
\newcommand{\ssr}{\\ & SSR: }
\newcommand{\sep}{ \\  \hline}

\begin{longtable}{rl}
\hline
II 
  \nf $q  r  s$ 
  \sd  $v_0  v_1  v_2$ 
  \ssr $1$
  \sep
III  
  \nf $a  q  p + b  r  p$ 
  \sd $v_0  v_1  v_2+ v_0 v_3  v_4$ 
  \ssr $2$
  \sep 
IV  
  \nf $a  p  r + b  r  p + c  p  q$ 
  \sd $v_0 v_1  v_2+v_0 v_3  v_4+v_1 v_3  v_5$ 
  \ssr $3$
  \sep 
V 
  \nf $a  b  c + p  r  q$ 
  \sd $v_0 v_1  v_2+v_3 v_4  v_5$ 
  \ssr $2$
  \sep 
VI 
  \nf $a q p +b r p+ c s p$ 
  \sd $v_0 v_1 v_2 + v_0 v_3 v_4 + v_0 v_5 v_6$
  \ssr $3$
  \sep 
VII 
  \nf $q r s+ a q p +b r p+ c s p$ 
  \sd $v_0 v_1 v_2 + v_0 v_3 v_4 + v_5 v_6 (v_0+ \cdots +v_6)$ 
  \ssr $3$
  \sep 
VIII 
  \nf $a b c + q r s+ a q p $
  \sd $v_0 v_1 v_2 + v_0 v_3 v_4 + v_3 v_5 v_6$ 
  \ssr $3$
  \sep 
IX 
  \nf $a b c + q r s+ a q p +b r p$
  \sd $v_0 v_1  v_2 +v_3 v_4  v_5  +v_6(v_0+\cdots +v_5)m_0$ 
  \ssr $3$ 
  \sep
X 
  \nf $a b c + q r s + a q p + b r p + c s p$
  \sd $v_0 v_1  v_2 +v_3 v_4  v_5  +v_6 (v_0+\cdots +v_6) m_0+  m_1 m_2  m_3$ 
  \ssr $4$
  \sep 
XI 
  \nf $a q p+ b r p + c s p+  c r t $ 
  \sd $v_0v_1v_2 + v_0v_3v_4 + v_0v_5v_6 + v_1v_3v_7$
  \ssr $4$
  \sep 
XII 
  \nf $q r s+ a q p+  b r p + c s p+  c r t$ 
  \sd $(a-s)qp + (q-c)(p+r)s + (t+s)cr + brp$ 
  \ssr $4$
  \sep 
XIII 
  \nf $a b c+ q r s+ a q p+  c r t$ 
  \sd $v_0v_1v_2 + v_0v_3v_4 + v_1v_5v_6 + v_3v_5v_7$ 
  \ssr $4$
  \sep 
XIV 
  \nf $a b c+ q r s+ a q p+  b r p + c r t$ 
  \sd $ab(c-p) + (a-r)(b+q)p + rq(p-s) + crt$ 
  \ssr $4$
  \sep 
XV 
  \nf $a b c+ q r s+ a q p+  b r p + c s p+  c r t$ 
  \sd $crt + aqp + (b+s)(r-c)p + (a+p)bc + (p+q)rs$ 
  \ssr $5$
  \sep 
XVI 
  \nf $a q p+  b s t+ c r t$ 
  \sd $v_0v_1v_2 + v_0v_3v_4 + v_5v_6v_7$ 
  \ssr $3$
  \sep 
XVII 
  \nf $a q p+  b r p + b s t+ c r t$ 
  \sd $v_0v_1v_2 + v_0v_3v_4 + v_1v_3v_5 + v_2v_6v_7 $ 
  \ssr $4$
  \sep 
XVIII 
  \nf $q r s+ a q p+  b r p + b s t+ c r t$ 
  \sd $(t-r)bs + crt + r(p-s)(b-q) + (a-r)qp$ 
  \ssr $4$
  \sep 
XIX 
  \nf $a q p+  b r p + c s p+  b s t+ c r t$ 
  \sd $v_0v_1v_2 + v_3v_4v_5 + v_6v_7(v_0+\cdots+v_5)$ 
  \ssr $3$
  \sep 
XX 
  \nf $q r s+ a q p+  b r p + c s p+  b s t+ c r t$ 
  \sd $(r+s)(t-r)b + (r+s)(r+p)(c-q) + (a-r-s)qp + r(b-c)(s-p-t)$ 
  \ssr $4$
  \sep 
XXI 
  \nf $a b c+ q r s+ a q p+  b s t$ 
  \sd $v_0v_1v_2 + v_0v_3v_4 + v_1v_5v_6 + v_3v_5v_7$ 
  \ssr $4$
  \sep
XXII 
  \nf $a b c+ q r s+ a q p+  b r p + b s t+ c r t$ 
  \sd $(a+r)(b+2q)(p-c+\tfrac{1}{2}s) + cr(t-3b-2q) + bs(\tfrac{1}{2}a+\tfrac{3}{2}r+t)$ \\
   & ${}\qquad + (b+q)(a+2r)(p-2c+s) $ 
  \ssr $4$
  \sep
XXIII 
  \nf $a b c+ q r s+ a q p+  b r p + c s p+  b s t+ c r t$ 
  \sd $v_0v_1v_2 + v_3v_4v_5 + v_6v_7(v_0+\cdots+v_5) + l_0l_1l_2$ 
  \ssr $4$ 
  \sep
\end{longtable}

\section*{Acknowledgements} We like to thank M. Brion who pointed out Remark \ref{Brion}. We also thank L. Oeding, G. Ottaviani, and J. Weyman for having provided us very good references, and for fruitful discussions.

The third author was partially supported by CIMA -- Centro de Investiga\c{c}\~{a}o em Matem\'{a}tica e Aplica\c{c}\~{o}es, Universidade de \'{E}vora, project UID/MAT/04674/2019 (Funda\c{c}\~{a}o para a Ci\^{e}ncia e Tecnologia). 

Part of this work was done while the third author was visiting Trento University, he wishes to thank the Mathematics Department for their hospitality.

\providecommand{\bysame}{\leavevmode\hbox to3em{\hrulefill}\thinspace}

\end{document}